\documentclass[12pt,a4paper]{amsart}

\usepackage{amsmath,amsthm,amssymb,latexsym,a4wide,tikz,multicol,tikz-cd}
\usepackage{arydshln,multirow,soul}
\usepackage{xcolor,enumerate}
\usepackage[active]{srcltx}
\usepackage{mathtools,stmaryrd}
\tikzset{font=\small}
\usetikzlibrary{matrix,arrows}
\usetikzlibrary{positioning}
\usetikzlibrary{patterns}

\newtheorem{theorem}{Theorem}[section]
\newtheorem{corollary}[theorem]{Corollary}
\newtheorem{proposition}[theorem]{Proposition}

\newtheorem{example}[theorem]{Example}
\newtheorem{lemma}[theorem]{Lemma}
\newtheorem{question}[theorem]{Question}

\theoremstyle{definition}
\newtheorem{remark}[theorem]{Remark}
\newtheorem{definition}[theorem]{Definition}

\numberwithin{equation}{section}

\title[Partial actions and proper extensions of restriction semigroups]{Partial actions and proper extensions of two-sided restriction semigroups}

\author{Mikhailo Dokuchaev}
\address{M. Dokuchaev: Insituto de Matem\'atica e Estat\'istica, Universidade de S\~ao Paulo,  Rua do Mat\~ao, 1010, S\~ao Paulo, SP,  CEP: 05508--090, Brazil}
\email{dokucha@gmail.com}

\author{Mykola Khrypchenko}

\address{M. Khrypchenko: Departamento de Matem\'atica, Universidade Federal de Santa Catarina, Campus Reitor Jo\~ao David Ferreira Lima, Florian\'opolis, SC,  CEP: 88040--900, Brazil  \and Centro de Matem\'atica e Aplica\c{c}\~oes, Faculdade de Ci\^{e}ncias e Tecnologia, Universidade Nova de Lisboa, 2829-516 Caparica, Portugal}
\email{nskhripchenko@gmail.com}

\author{Ganna Kudryavtseva}
\address{G. Kudryavtseva: University of Ljubljana,
	Faculty of Civil and Geodetic Engineering, Jamova cesta~2, SI-1000 Ljubljana, Slovenia \and Institute of Mathematics, Physics and Mechanics, Jadranska ulica 19, SI-1000 Ljubljana, Slovenia}
\email{ganna.kudryavtseva\symbol{64}fgg.uni-lj.si}


\begin{document}
	\maketitle
	
	\begin{abstract} 
		We prove a structure result on proper extensions of two-sided restriction semigroups in terms of partial actions, generalizing respective results for monoids and for inverse semigroups and upgrading the latter. We introduce and study several classes of partial actions of two-sided restriction semigroups that generalize partial actions of monoids and of inverse semigroups. We establish an adjunction between the category ${\mathcal{P}}(S)$ of proper extensions of a restriction semigroup $S$ and a category ${\mathcal{A}}(S)$ of partial actions of $S$ subject to certain conditions going back to the work of O'Carroll. In the category ${\mathcal{A}}(S)$, we specify two isomorphic subcategories, one being reflective and the other one coreflective, each of which is equivalent to the category ${\mathcal{P}}(S)$.
	\end{abstract}
	
	\section{Introduction}
	Two-sided restriction semigroups, or simply restriction semigroups, also known as weakly $E$-ample semigroups, are non-involutive generalizations of inverse semigroups. These are algebras $(S; \cdot\, , ^*, ^+)$ of type $(2,1,1)$ such that $(S; \cdot)$ is a semigroup and the unary operations $^*$ and $^+$ resemble the operations $s\mapsto s^{-1}s$ and $s\mapsto ss^{-1}$, respectively, on an inverse semigroup. The latter operations, for inverse semigroups of partial bijections, are precisely the operations of taking the domain and the range idempotents of $s$.  Restriction semigroups and their one-sided analogues arise naturally from various sources and have been widely studied by many authors, see~\cite{CG,Gomes05,Gomes06,GG00,GHSz17,GH09,H2,Jones16,Kud15,Kud19,Sz12} and references therein.
	
	Many results on inverse semigroups  admit extensions to restriction semigroups. These extensions sometimes include special cases, which are important general theorems about restriction semigroups but do not restrict to such theorems for inverse semigroups. This shows that restriction semigroups have their own rich theory, which is not a mere generalization of the inverse semigroup theory. 
		
	Proper restriction semigroups \cite{CG, GHSz17, Jones16,  Kud15, Kud19,Sz12} generalize the well-established and widely studied class of $E$-unitary inverse semigroups. To illustrate the thesis of the previous paragraph we mention almost perfect restriction semigroups \cite{Jones16} (called ultra proper restriction semigroups in \cite{Kud15}) which form a subclass of proper restriction semigroups and play a central role in a number of structure results, but it does not reduce to a class of inverse semigroups of a similar importance.
	
	A restriction semigroup $S$ is proper if and only if its minimum reduced restriction semigroup congruence $\sigma$ coincides with the compatibility relation $\sim$ (see Section \ref{s:prelim} for details). In other words, $S$ is proper if and only if 
for all $s,t\in S$: $\sigma^{\sharp}(s) = \sigma^{\sharp}(t) \Leftrightarrow s\sim t$, where $\sigma^{\sharp}\colon S\to S/\sigma$ is the canonical projection. Proper extensions of restriction semigroups generalize proper restriction semigroups and arise from surjective morphisms $\psi\colon S\to T$ between restriction semigroups (where $T$ is not necessarily reduced) satisfying a similar requirement:  for all $s,t\in S$: $\psi(s)=\psi(t)\Leftrightarrow s\sim t$.  At the same time proper extensions of restriction semigroups generalize idempotent pure extensions of inverse semigroups \cite{OC77, LMS}.

	The work \cite{CG} by Cornock and Gould establishes a structure result for proper restriction semigroups in terms of partial actions of monoids on semilattices, generalizing the formulation of McAlister's $P$-theorem due to Kellendonk and Lawson \cite{KL} which stems from Petrich and Reilly \cite{PR79}. The partial actions therein are in general non-globalizable and thus one can not apply globalization to reformulate the result  in terms of actions. This is in contrast with the situation with partial group actions in the structure result for $E$-unitary inverse semigroups \cite{KL} which are always globalizable.
	
	In the present paper we provide an extension, as well as an appropriate framework for such an extension, of the result on the structure of proper restriction semigroups in terms of partial actions \cite{CG} mentioned above. We replace proper restriction semigroups thereof  by proper extensions of restriction semigroups. At the same time, we extend and upgrade the partial action variant \cite{Kh17} of  O'Carroll's result on the structure of  idempotent pure extensions of inverse semigroups \cite{OC77}.  Furthermore,  we describe the relationship between the category of proper extensions of a restriction semigroup $S$ and certain categories of partial actions of $S$.

	Our results involve  proper extensions and partial actions of  restriction semigroups, which we define and discuss prior to formulating the results. We adapt the notion of a proper extension from the work by Gomes~\cite{Gomes05} (where proper extensions of left restriction semigroups were studied, from a different perspective). Our notions of a partial action of a restriction semigroup and its corresponding premorphism are influenced by similar notions existing in the contexts of left restriction semigroups \cite{EQF05, Gomes06, GH09} and of inverse semigroups \cite{GH09, KL, LMS, MR77}. In fact, we  introduce several classes of partial actions (more precisely, of premorphisms associated to partial actions) of restriction semigroups by partial bijections that arise naturally in analogy with suitable classes of partial actions of inverse semigroups (in particular, extending them). Some of these classes do not have precursors in the literature which is notably the class of locally strong premorphisms that extend  premorphisms between inverse semigroups from \cite{LMS}.
	
	We now briefly describe the structure of the paper and highlight its main results.  Section~\ref{s:prelim} contains prerequisites. In Section~\ref{s:premorphisms} we define premorphisms between restriction semigroups and discuss their connection with partial actions by partial bijections. We introduce  several classes of premorphisms: order-preserving premorphisms,  strong and locally strong premorphisms, multiplicative and locally multiplicative premorphisms. In Section \ref{s:structure} we bring into discussion proper extensions of restriction semigroups and the partial action product $Y\rtimes^q_{\varphi} S$ of a semilattice $Y$ acted partially upon by a restriction semigroup $S$, subject to certain conditions (conditions (A1)--(A4)), with respect to a semilattice morphism $q\colon Y\to P(S)$. We prove that $Y\rtimes^q_{\varphi} S$ is a restriction semigroup and, moreover, a proper extension of~$S$. We then define the two {underlying premorphisms, $\widehat{\psi}$ and  $\widetilde{\psi}$, of a proper extension $\psi\colon T\to S$ and prove Theorem~\ref{th:isom_epsilon} stating that given a proper extension $\psi\colon T\to S$, the semigroup $T$ decomposes into a partial action product of $P(T)$ acted upon (by means of either $\widehat{\psi}$ or $\widetilde{\psi}$) by $S$. Further, in Section \ref{s:classes} we study the relationship between the properties of $\psi$, $\widehat{\psi}$ and  $\widetilde{\psi}$. Depending on these properties, we single out several natural classes of proper extensions: order-proper extensions, extra proper extensions (which generalize extra proper restriction semigroups \cite{CG, Kud15}) and perfect extensions (which generalize perfect \cite{Jones16}, or ultra proper \cite{Kud15},  restriction semigroups). In Section \ref{s:categorical} we introduce three categories of partial actions of $S$ (more precisely, of premorphisms $S\to {\mathcal{I}}(X)$): the category of `most general' partial actions ${\mathcal{A}}(S)$ from which a proper extension of $S$ can be constructed, which goes back to the work of O'Carroll \cite{OC77}, and two its subcategories, ${\widehat{\mathcal{A}}}(S)$ and ${\widetilde{\mathcal{A}}}(S)$, subject to certain additional restrictions, so that  ${\widehat{\mathcal{A}}}(S)$ and ${\widetilde{\mathcal{A}}}(S)$ contain as objects premorphisms of the form $\widehat{\psi}$ and $\widetilde{\psi}$, respectively. 
		In Theorem \ref{th:adjunctions} we prove that the categories ${\widehat{\mathcal{A}}}(S)$ and ${\widetilde{\mathcal{A}}}(S)$ are isomorphic and, furthermore, that ${\widehat{\mathcal{A}}}(S)$ is a reflective subcategory and ${\widetilde{\mathcal{A}}}(S)$ a coreflective subcategory of the category $\mathcal{A}(S)$. Afterwards, we establish an equivalence of the category ${\mathcal{P}}(S)$ of proper extensions of $S$ and the category ${\widehat{\mathcal{A}}}(S)$ (see Theorem \ref{th:equivalence_of_cats}) and formulate its consequences. We also address a question: which property of morphisms of the category $\mathcal{A}(S)$ corresponds to surjectivity of morphisms of the subcategory ${\mathcal{P}}(S)$? We find condition (M3) which at first glance may seem too weak but does the work. We thus specialize the obtained results to respective facts about the subcategory ${\mathcal{P}}_s(S)$ of ${\mathcal{P}}(S)$ whose morphisms are surjective and the subcategory $\mathcal{A}_s(S)$ of $\mathcal{A}(S)$ whose morphisms satisfy condition~(M3). The paper concludes with  Section \ref{s:f} where proper extensions $\psi\colon T\to S$ such that $\psi^{-1}(s)$  has a maximum element for each $s\in S$ are touched upon. If the premorphism $\widetilde{\psi}$ is order-preserving such extensions generalize $F$-restriction monoids \cite{Jones16,Kud15} and at the same time $F$-morphisms from \cite{LMS}. If, in addition, $\widetilde{\psi}$ is locally strong, the obtained class is analogous to that of $F\!A$-morphisms from~\cite{Gomes06} and generalizes extra proper $F$-restriction monoids from \cite{Kud15} as well as $F$-morphisms from \cite{LMS}.

		\section{Preliminaries}\label{s:prelim}
		\subsection{Restriction semigroups} In this section we recall the definition and basic properties of two-sided restriction semigroups. 
		More details can be found in \cite{Cornock_phd, G,H2}.
		
		A {\em left restriction semigroup} is an algebra $(S; \cdot \,, ^+)$, where $(S;\cdot)$ is a semigroup and $^+$ is a unary operation satisfying the following identities:
		\begin{equation}\label{eq:axioms:plus}
		x^+x=x, \,\,\, x^+y^+=y^+x^+, \,\,\, (x^+y)^+=x^+y^+, \,\,\, (xy)^+x=xy^+.
		\end{equation}
		Dually, a {\em right restriction semigroup} is an algebra $(S; \cdot \,, ^*)$, where $(S;\cdot)$ is a semigroup and $^*$ is a unary operation satisfying the following identities:
		\begin{equation}\label{eq:axioms:star}
		xx^*=x, \,\,\, x^*y^*=y^*x^*, \,\,\, (xy^*)^*=x^*y^*, \,\,\, y(xy)^*=x^*y.
		\end{equation}
		A {\em two-sided restriction semigroup}, or just a {\em restriction semigroup}, is an algebra $(S; \cdot\, , ^*, ^+)$, where $(S;\cdot\, , ^+)$ is a left restriction semigroup, $(S;\cdot\, , ^*)$ is a right restriction semigroup, and the operations $^*$ and $^+$ are connected by the following identities:
		\begin{equation}\label{eq:axioms:common}
		(x^+)^*=x^+,\,\,\, (x^*)^+=x^*.
		\end{equation}
		In this paper we always consider restriction semigroups  as $(2,1,1)$-algebras.  It is immediate from the definition that restriction semigroups form a variety of signature $(2,1,1)$. 
		A   {\em morphism of restriction semigroups} is, by definition,  a $(2,1,1)$-morphism, that is, it preserves the multiplication and the unary operations $^*$ and $^+$.

		Let $S$ be a restriction semigroup. From \eqref{eq:axioms:common} it follows that 
		$$
		\{s^*\colon s\in S\}=\{s^+\colon s\in S\}.
		$$
		This set, denoted by $P(S)$, is closed with respect to the multiplication. Furthermore, it is a semilattice and  $e^*=e^+=e$ holds for all $e\in P(S)$.  It is called the {\em semilattice of projections} of $S$ and its elements are called {\em projections}. A projection is necessarily an idempotent, but a restriction semigroup may contain idempotents that are not projections.

		We will often use the following equalities: 
		\begin{equation}\label{eq:mov_proj}
		es = s(es)^* \,\, \text{and} \,\, se = (se)^+s \,\, \text{for all} \,\, s\in S  \,\, \text{and} \,\, e\in P(S);
		\end{equation}
		\begin{equation}\label{eq:rule1}
		(se)^* = s^*e \,\, \text{and} \,\, (es)^+ = es^+ \,\, \text{for all} \,\, s\in S  \,\, \text{and} \,\, e\in P(S);
		\end{equation}
		\begin{equation}\label{eq:consequences}
		(st)^*=(s^*t)^*,\,\,\, (st)^+=(st^+)^+ \,\, \text{for all} \,\, s,t\in S.
		\end{equation}
		
		Inverse semigroups can be looked at as restriction semigroups if one puts $s^*=s^{-1}s$ and $s^+=ss^{-1}$ for all elements $s$. If $S$ is an inverse semigroup then, necessarily, $P(S)=E(S)$.
		
		\begin{remark}\label{rem:morphisms_inv} 
			When considering inverse semigroups as restriction semigroups, by a morphism between them it is natural to mean a $(\cdot\,, ^*,^+)$-morphism. However, a map between inverse semigroups that preserves the multiplication~$\cdot$ necessarily preserves also the inverse operation $^{-1}$.  It follows that $(\cdot\,, ^*,^+)$-morphisms between inverse semigroups coincide with  semigroup morphisms and thus also with $(\cdot\,, ^{-1})$-morphisms.
		\end{remark}
		
		  Elements $s,t\in S$ are said to be {\em compatible}, denoted $s\sim t$, if $st^* = ts^*$ and $t^+s = s^+t$.  Note that in the symmetric inverse semigroup, two elements are compatible if an only if they agree on the intersection of their domains. The following properties related to the natural partial order and the compatibility relation will be used throughout the paper.

		\begin{lemma}\label{lem:lem1} Let $S$ be a restriction semigroup and $s,t\in S$. Then:
			\begin{enumerate}[(1)]
				\item \label{i:b1} $s\leq t$ if and only if  $s=tf$ for some  $f\in P(S)$;
				\item \label{i:b2} $s\leq t$ if and only if $s=ts^*$ if and only if $s=s^+t$;
				\item \label{i:b5} $s\leq t$ implies $su\leq tu$ and $us\leq ut$ for all $u\in S$;
				\item\label{i:b6}  $s\leq t$ implies $s^*\leq t^*$ and $s^+\leq t^+$; 
				\item \label{i:b61} $s\leq t$ implies $s\sim t$;
				\item \label{i:b62} $s,t\leq u$ for some $u\in S$ implies $s\sim t$;
				\item\label{i:b7}  $s\sim t$ and $s^*\leq t^*$ (or $s^+\leq t^+$) imply $s\leq t$;
				\item\label{i:b8} $s\sim t$ and $s^*=t^*$ (or $s^+=t^+$) imply $s=t$.
			\end{enumerate}
		\end{lemma}
		
		A {\em reduced restriction semigroup} is a restriction semigroup $S$ for which $|P(S)|=1$. Then, necessarily, $S$ is a monoid and $P(S)=\{1\}$ so that  \mbox{$s^*=s^+=1$} holds for all $s\in S$. On the other hand, any monoid $S$ can be endowed with the structure of a restriction semigroup by putting $s^*=s^+= 1$ for all $s\in S$. Hence reduced restriction semigroups can be identified with monoids. 
		
		For further use, we record the following facts.
		
		\begin{lemma}\label{lem:aux14} \mbox{}
		\begin{enumerate}
		\item[(1)] Let $s,t\in S$ and $e\in P(S)$. If $e\leq (st)^*$ then $(te)^+\leq s^*$. 
		\item [(2)] Let $q,r,s,t \in S$ be such that $s\sim t$, $q\leq s$ and $r\leq t$. Then $q\sim r$.
		\end{enumerate} 
		\end{lemma}
		
		\begin{proof} 
			 (1) Note that $s^*t\leq t$ yields $s^*t = t(s^*t)^* = t (st)^*$. Since $e\leq (st)^*$, we have $(te)^+ \leq (t(st^*))^+= (s^*t)^+\leq s^*$.
			
			(2) Since $q\leq s$ and $r\leq t$ there are $e,f\in P(S)$ such that $q=se$ and $r=tf$. Applying~\eqref{eq:rule1} and commuting projections we have:
			$$qr^*= (se)(tf)^* = set^*f = st^*ef.$$
			Dually, $rq^*=ts^*fe$. In view of $ts^*=st^*$ this yields $qr^*=rq^*$. In a similar way one shows that $r^+q=q^+r$.
		\end{proof}
		
		Let $\sigma$ denote the least congruence on a restriction semigroup $S$ that identifies all elements of $P(S)$.  Each of the following statements is equivalent to $s \, {\mathrel{\sigma}}\, t$  (see \cite[Lemma 1.2]{CG}):
		\begin{enumerate}
			\item[(i)] there is $e\in P(S)$ such that $es = et$; 
			\item[(ii)] there is $e \in P(S)$ such that $se = te$.
		\end{enumerate}
		Clearly, $s\sim t$ implies $s\mathrel{\sigma} t$.
		A restriction semigroup $S$ is called {\em proper} if the following two conditions hold: 
		$$ \text{for all } s,t\in S: \text{ if } s^* =t^* \text{ and } \, s \, {\mathrel{\sigma}} \, t \, \text{ then } s=t;
		$$
		$$ \text{for any } s,t\in S: \text{ if } s^+ =t^+ \text{ and } \,  s \, {\mathrel{\sigma}} \, t \, \text{ then } s=t.
		$$
		It is known that a restriction semigroup is proper if and only if $\sim \, = \sigma$. Indeed, suppose that $S$ is proper and let $s,t\in S$ be such that  $s \mathrel{\sigma} t$. Then $st^*\mathrel{\sigma} ts^*$ and $(st^*)^* = (ts^*)^*$. It follows that $st^*=ts^*$. Dually we have $t^+s=s^+t$ whence $s\sim t$. Conversely, suppose that $\sim \,= \sigma$ and that $s,t\in S$ are such that $s \mathrel{\sigma} t$ and $s^*=t^*$. Then $s\sim t$. From Lemma \ref{lem:lem1}\eqref{i:b8} we get $s=t$, so that $S$ is proper.
		
		If $(A, \leq)$ is a poset and $B\subseteq A$, the {\em downset} of $B$ or the {\em order ideal} generated by $B$ is the set
		$$
		B^{\downarrow} = \{a\in A\colon a\leq b \text{ for some } b\in B\}. 
		$$
		If $b\in A$ we write $b^{\downarrow}$ for $\{b\}^{\downarrow}$ and call this set the {\em principal order ideal} generated by $b$.
		
		Recall that the {\em Munn semigroup} $T_Y$ of a semilattice $Y$ is the inverse subsemigroup of ${\mathcal{I}}(Y)$ consisting of all order isomorphisms between principal order ideals of $Y$.  If $Y$ has a top element, $T_Y$ is a monoid whose identity is ${\mathrm{id}}_Y$, the identity map on $Y$.  
		
		The {\em Munn representation} of a restriction semigroup $S$ \cite{Jones16, Kud15, Kud19} is a morphism $\alpha\colon S\to T_{P(S)}$, $s\mapsto \alpha_s$, given, for each $s\in S$, by
		\begin{equation*}\label{eq:munn}
		\operatorname{\mathrm{dom}}\alpha_s = (s^*)^{\downarrow} \,  \text{ and } \, \alpha_s(e) = (se)^+, \, e\in \operatorname{\mathrm{dom}}\alpha_s.
		\end{equation*}
		
		This generalizes the Munn representation of an inverse semigroup \cite[Theorem 5.2.8]{Lawson:book}.
		
		\section{Premorphisms and partial actions of restriction semigroups}\label{s:premorphisms}
		\subsection{Premorphisms and their corresponding partial actions}
		\begin{definition} \label{def:pm} Let $S$ and $T$ be  restriction semigroups.
			A map $\varphi\colon S\to T$ is called a {\em premorphism} if the following conditions hold:
			\begin{enumerate}
				\item[{$\mathrm{(PM1)}$}] $\varphi(s)\varphi(t) \leq \varphi(st)$ for all $s,t\in S$;
				\item[{$\mathrm{(PM2)}$}] $\varphi(s)^* \leq \varphi(s^*)$ for all $s\in S$;
				\item[{$\mathrm{(PM3)}$}] $\varphi(s)^+ \leq \varphi(s^+)$ for all $s\in S$.
			\end{enumerate}
		\end{definition}
		
		Conditions (PM2) and (PM3) are included into the definition  because proper extensions of restriction semigroups induce maps satisfying all the conditions (PM1), (PM2) and (PM3) (see Proposition \ref{prop:psi_hat_tilde}). Note that  premorphisms between one-sided restriction semigroups that have been considered in the literature \cite{EQF05, Gomes06, GH09} are also required to satisfy (PM2) or (PM3).
		
		\begin{remark} Let $S$ be monoid and $T$ a restriction monoid.  If $\varphi\colon S\to T$ a map satisfying (PM1) and $\varphi(1)=1$ (such maps were considered in \cite{Kud19}),
			it clearly satisfies (PM2) and (PM3) and thus is a premorphism. Conversely, one can show that if $\varphi\colon S\to T$ is a premorphism and $T$ is generated by $\varphi(S)$ then $\varphi(1)=1$ also holds.
		\end{remark}
		
		\begin{lemma}\label{lem:proj_identity} Let $\varphi\colon S\to T$ be a premorphism and $e\in P(S)$. Then $\varphi(e)\in P(T)$. In particular, if $T={\mathcal{I}}(X)$ then $\varphi(e)(x) = x$ for all $e\in P(S)$ and $x\in X$.
		\end{lemma}
		
		\begin{proof} From (PM2) we have $\varphi(e)^*\leq \varphi(e)$ whence $\varphi(e)^* = \varphi(e)(\varphi(e)^*)^* = \varphi(e)\varphi(e)^* = \varphi(e)$.  It follows that $\varphi(e)\in P(T)$.
		\end{proof}
		
		\begin{lemma} \label{lem:prem_compatibility} Let $\varphi\colon S\to T$ be a premorphism. Then $s\sim t$ in $S$ implies $\varphi(s)\sim \varphi(t)$ in~$T$. 
		\end{lemma}
		
		\begin{proof} Let $s\sim t$ in $S$. Then $\varphi(s)\varphi(t)^*\leq \varphi(s)\varphi(t^*)\leq \varphi(st^*)$. From $st^*=ts^*$ we have that $\varphi(t)\varphi(s)^*\leq \varphi(ts^*)=\varphi(st^*)$. It follows that $\varphi(s)\varphi(t)^* \sim \varphi(t)\varphi(s)^*$. But since $(\varphi(s)\varphi(t)^*)^* = (\varphi(t)\varphi(s)^*)^*$ we have $\varphi(s)\varphi(t)^* = \varphi(t)\varphi(s)^*$. Similarly one shows that $\varphi(s)^+\varphi(t) = \varphi(t)^+\varphi(s)$.
		\end{proof}
		
		Recall that a premorphism $\varphi\colon S\to T$ between inverse semigroups is known \cite{KL, LMS, MR77} to be a map satisfying  condition (PM1) and also the condition 
		\begin{enumerate}
			\item[(Inv)] $\varphi(s^{-1})=\varphi(s)^{-1}$ for all $s\in S$.
		\end{enumerate}
		
		\begin{remark}\label{rem:inverse} Let $S$ and $T$ be inverse semigroups. Then, under the presence of (PM1) and (Inv), conditions (PM2) and (PM3) hold automatically.  So a premorphism between inverse semigroups in the sense of \cite{KL, LMS,MR77} is a premorphism in the sense of Definition~\ref{def:pm}.
		\end{remark}
		
		Throughout the paper, by a premorphism between restriction semigroups and, in particular, between inverse semigroups we mean a map satisfying conditions (PM1), (PM2) and (PM3). If $\varphi\colon S\to T$ is a premorphism between inverse semigroups satisfying also condition (Inv), we call it an {\em inverse premorphism}. Remark \ref{rem:inverse} tells us that inverse premorphisms between inverse semigroups are precisely the ones considered in the literature \cite{KL, LMS, MR77}. Note that premorphisms between inverse semigroups (in the sense of Definition \ref{def:pm}) are not required to satisfy condition (Inv) and, as can be shown, do not necessarily satisfy it (in contrast to the situation with morphisms, see Remark \ref{rem:morphisms_inv}).
		
		Let $S$ be a restriction semigroup and $X$ a set. A {\em left partial action of} $S$ on $X$ {\em by partial bijections} (or just a {\em left partial action of} $S$ {\em on} $X$)  is a partially defined map $S\times X \to X$, $(s,x)\mapsto s\cdot x$, such that:
		
		\begin{enumerate} [(i)]
			\item if $x\neq y$ and $s\cdot x$ and $s\cdot y$ are defined then $s\cdot x \neq s\cdot y$, for all $s\in S$ and $x,y\in X$ (each $s\in S$ acts by a partial bijection);
			\item  if $s\cdot x$ and $t\cdot (s\cdot x)$ are defined then $ts\cdot x$ is defined and $t\cdot (s\cdot x) = ts\cdot x$, for all $s,t\in S$ and $x\in X$ (to match condition (PM1));
			\item  if $s\cdot x$ is defined then $s^*\cdot x$ is defined and $s^*\cdot x=x$ (to match condition (PM2));
			\item if $s\cdot x$ is defined then $s^+\cdot (s\cdot x)$ is defined (to match condition (PM3)).
		\end{enumerate}
		
		It is clear by (iii) that for $p\in P(S)$ if $p\cdot x$ is defined then $p\cdot x=x$. Furthermore, (ii) and (iv) imply that if $s\cdot x$ is defined then $s^+\cdot (s\cdot x) = s\cdot x$.
		
		A {\em right partial action} of $S$ on $X$ is defined dually. 
		Namely, a {\em right partial action of} $S$ on $X$ {\em by partial bijections} is a partially defined map $X\times S \to X$, $(x,s)\mapsto x\circ s$, such that:
		
		\begin{enumerate}[(i$'$)]
			\item if $x\neq y$ and $x\circ s$ and $y\circ s$ are defined then $x\circ s \neq y\circ s$, for all $s\in S$ and $x,y\in X$;
			\item if $x\circ s$ and $(x\circ s)\circ t$ are defined then $x\circ st$ is defined and $(x\circ s)\circ t = x\circ st$, for all $s,t\in S$ and $x\in X$;
			\item if $x\circ s$ then $x\circ s^+$ is defined and $x\circ s^+=x$;
			\item  if $x\circ s$ is defined then $(x\circ s)\circ s^*$ is defined. 
		\end{enumerate}

		If $S\times X \to X$, $(s,x)\mapsto s\cdot x$, is a left partial action of $S$ on $X$, for each $x\in X$ we define $\varphi_s\in {\mathcal{I}}(X)$ where 
		$$\operatorname{\mathrm{dom}}\varphi_s =\{x\in X\colon s\cdot x \, {\text{ is defined}}\}
		$$
		and $\varphi_s(x) = s\cdot x$  for all $x\in \operatorname{\mathrm{dom}}\varphi_s$. Then the map $\varphi\colon S\to {\mathcal{I}}(X)$, $s\mapsto \varphi_s$, is a premorphism.
		Conversely, given a premorphism $\varphi\colon S\to {\mathcal{I}}(X)$, $s\mapsto \varphi_s$, it determines a left partial action $S\times X \to X$ so that $s\cdot x$ is defined if and only if $x\in \operatorname{\mathrm{dom}}\varphi_s$ and in the latter case $s\cdot x = \varphi_s(x)$. We have described a one-to-one correspondence between left partial actions of $S$ on $X$ and premorphisms $S\to {\mathcal{I}}(X)$. 
		
		A left partial action $S\times X \to X$, $(s,x)\mapsto s\cdot x$,  of $S$ on $X$ determines a right partial action $X\times S \to X$, $(x,s)\mapsto x\circ s$, of $S$ on $X$ as follows. Let $s\in S$ and $x\in X$. Then
		$$
		x\circ s {\text{ is defined if and only if }} x= s\cdot y {\text{ for some }} y\in X {\text{ in which case }} x\circ s =y.
		$$
		This establishes a one-to-one correspondence between left partial actions and right partial actions of $S$ on $X$. 
		Note  that $x\circ s$ is defined if and only $x\in \operatorname{\mathrm{ran}}\varphi_s = \operatorname{\mathrm{dom}}\varphi_s^{-1}$ and $x\circ s = \varphi_s^{-1}(x)$, for any $s\in S$ and $x\in X$. From this, it is easily seen that, e.g., $(s\cdot x)\circ s$ is defined if and only if $s\cdot x$ is defined and in the latter case $(s\cdot x)\circ s = \varphi_s^{-1}\varphi_s(x)=x$.
		
		Throughout the paper, we work with premorphisms rather than with partial actions. However, all classes of premorphisms under consideration can be interpreted as suitable classes of partial actions.

		\subsection{Strong and order-preserving premorphisms.}
		
		It is natural to look for a class of premorphisms $\varphi\colon S\to T$ between restriction semigroups which, when specialized to the case where both $S$ and $T$ are inverse semigroups, coincides with inverse premorphisms.  Since we are mostly interested in premorphisms that arise from partial actions by partial bijections, we focus on premorphisms from a restriction semigroup to an inverse semigroup.

		A condition, satisfied by a premorphism $\varphi\colon S\to T$, where $S$ is a group and $T$ a monoid, equivalent to (Inv) and expressed in the signature of one-sided restriction semigroups,  was found in \cite{H}, and respective partial actions had been introduced prior to that in \cite{MS04}. Such premorphisms were termed {\em strong} in \cite{H}, see also \cite{GH09}. Here we adapt these notions to the setting of premorphisms between (two-sided) restriction semigroups.

		\begin{definition} A premorphism $\varphi\colon S\to T$ between restriction semigroups is called {\em strong} if it satisfies the following conditions:
			\begin{enumerate}
				\item[${\mathrm{(Sr)}}$] $\varphi(s)\varphi(t) = \varphi(st)\varphi(t)^*$ for all $s,t\in S$;
				\item[${\mathrm{(Sl)}}$] $\varphi(s)\varphi(t) = \varphi(s)^+\varphi(st)$ for all $s,t\in S$.
			\end{enumerate}
		\end{definition}
		
		\begin{proposition} \label{prop:group_inv} Let $S,T$ be inverse semigroups and $\varphi\colon S\to T$ a premorphism.
			The following statements are equivalent:
			\begin{enumerate}[(1)]
				\item $\varphi$ satisfies condition (Sr);
				\item $\varphi$ satisfies condition (Sl);
				\item $\varphi$ is strong.
			\end{enumerate}
		\end{proposition}

		\begin{proof} Since the implications (3) $\Rightarrow$ (1) and  (3) $\Rightarrow$ (2) are obvious, it is enough to prove only the implications (1) $\Rightarrow$ (3) and  (2) $\Rightarrow$ (3). We will prove the first of them, the second one being proved similarly.
			By (Sr) and (PM2) for any $s\in S$ we have $\varphi(s^{-1})\varphi(s)=\varphi(s^*)\varphi(s)^*= \varphi(s)^*=\varphi(s)^{-1}\varphi(s)$. Multiplying this by $\varphi(s)$ on the left we obtain $\varphi(s)\varphi(s^{-1})\varphi(s)=\varphi(s)$.  Replacing $s$ by $s^{-1}$ we get $\varphi(s^{-1})\varphi(s)\varphi(s^{-1})=\varphi(s^{-1})$ whence $\varphi(s^{-1})=\varphi(s)^{-1}$. Finally, taking inverses of both sides of (Sr) and replacing $t^{-1}$ by $s$ and $s^{-1}$ by $t$ we get (Sl).
		\end{proof}

		\begin{definition} A premorphism $\varphi\colon S\to T$ between restriction semigroups is called {\em order-preserving} if the following condition holds:
			\begin{enumerate}
				\item[(OP)] $s\leq t$ implies $\varphi(s)\leq\varphi(t)$, for all $s,t\in S$. 
			\end{enumerate}
		\end{definition} 
		
		Note that a morphism, as well as a premorphism from a monoid, is automatically order-preserving.
		In view of Remark \ref{rem:inverse}, the following is an immediate consequence of Proposition \ref{prop:group_inv} and a result proved in~\cite[Proposition 8.7]{Kud19}.
		
		\begin{theorem}\label{th:strong_order}
			Let $\varphi\colon S\to T$ be a premorphism between inverse semigroups. Then each of the conditions of Proposition \ref{prop:group_inv} is equivalent to $\varphi$ being inverse and order-preserving. 
		\end{theorem}
		
		The condition of being order-preserving in the statement above is not redundant (for an example of an inverse but not order-preserving premorphism see \cite[page~216]{LMS}).
		
		\begin{corollary} \label{cor:strong_groups} A premorphism $\varphi\colon S\to T$, where $S$ is a group and $T$ an inverse semigroup  is strong if and only if it is inverse.
		\end{corollary}
		
		Note that strong premorphisms in the context of one-sided restriction semigroups appear, explicitely or implicitely, in the study of prefix expansions of unipotent monoids~\cite{FGG99}, generalized prefix expansion of weakly left ample semigroups \cite{Gomes06} and partial actions of left restriction semigroups \cite{GH09}. In addition, a specialization of the main construction of \cite{Kud19} to strong premorphisms coincides with the well-known Birget-Rhodes prefix expansion of a group (see \cite[Corollary 8.11]{Kud19}).
		
		\subsection{Locally strong premorphisms}\label{subs:locally_strong}

		Let $\varphi\colon S\to T$ be a premorphism between restriction semigroups.
		We introduce the following   conditions:
		
		\vspace{0.1cm}
		
		\begin{enumerate}
			\item[(LSr)] $\varphi(st^+)\varphi(t) = \varphi(st)\varphi(t)^*$ for all $s,t\in S$;
			\item[(LSl)] $\varphi(s)\varphi(s^*t) = \varphi(s)^+\varphi(st)$ for all $s,t\in S$.
		\end{enumerate}
		
		We now observe that these conditions admit  the following `local' equivalents:
		
		\begin{enumerate}
			\item[(LSr$'$)] $\varphi(s)\varphi(t) = \varphi(st)\varphi(t)^*$ for all  $s,t\in S$  such that  $s^* \leq t^+$;
			\item[(LSl$'$)] $\varphi(s)\varphi(t) = \varphi(s)^+\varphi(st)$ for all $s,t\in S$ such that  $t^+ \leq s^*$.
		\end{enumerate}

		\begin{lemma}\label{prop:inverse} Conditions (LSr) and (LSl) are equivalent to conditions (LSr\,$'$\!) and (LSl\,$'$\!),  respectively.
		\end{lemma}
		
		\begin{proof} (1) It is immediate that (LSr) implies (LSr$'$).  Conversely, putting, for $s,t\in T$, $u=st^+$, we get $u^*\leq t^+$. Applying
			(LSr$'$) to $u$ and $t$ we obtain (LSr). The other equivalence is proved dually. 
		\end{proof}
		
		\begin{definition} A premorphism $\varphi\colon S\to T$ between restriction semigroups will be called 
			{\em locally strong}, if it satisfies conditions (LSr) and (LSl).
		\end{definition}

		It is immediate that conditions  (LSr$'$) and  (LSl$'$), respectively, imply conditions 
		\begin{enumerate}
			\item[(LSr$''$)] $\varphi(s)\varphi(t) = \varphi(st)\varphi(t)^*$ for all  $s,t\in S$  such that  $s^* = t^+$;
			\item[(LSl$''$)] $\varphi(s)\varphi(t) = \varphi(s)^+\varphi(st)$ for all $s,t\in S$ such that  $t^+ = s^*$.
		\end{enumerate}
		\begin{question}\label{question1} {\em Are conditions (LSr$'$) and  (LSl$'$) equivalent to conditions (LSr$''$) and  (LSl$''$), respectively?}
		\end{question}
		
		The following statement implies that, whenever $S$ and $T$ are inverse semigroups, the answer to Question \ref{question1} is affirmative.

		\begin{proposition} Let $\varphi\colon S\to T$ be a premorphism between inverse semigroups. Then each of the conditions (LSr\,$'$\!),  (LSl\,$'$\!),  (LSr\,$''$\!) and  (LSl\,$''$\!) is equivalent to (Inv). Consequently, $\varphi$ is locally strong if and only if it is inverse. \end{proposition}
		
		\begin{proof}  We prove the implications (Inv) $\Rightarrow$  (LSr$'$) $\Rightarrow$  (LSr$''$) $\Rightarrow$ (Inv). We first assume that $\varphi$ satisfies~(Inv). Then, for all  $s,t\in S$ such that  $s^*\leq t^+$, applying (PM1), we have:
			\begin{multline*}
			\varphi(s)\varphi(t)=\varphi(s)\varphi(t)\varphi(t)^* \leq \varphi(st)\varphi(t)^* = \varphi(st)\varphi(t)^{-1} \varphi(t) = \varphi(st)\varphi(t^{-1}) \varphi(t)  \\
			\leq  \varphi(stt^{-1})\varphi(t) = \varphi(s)\varphi(t).
			\end{multline*}
			It follows that (LSr$'$) holds. The implication (LSr$'$) $\Rightarrow$  (LSr$''$) is trivial.
			
			(LSr$''$) $\Rightarrow$ (Inv) We now assume that  (LSr$''$) holds.  For each $s\in S$ we have:
			$$
			\varphi(s)\varphi(s^{-1})\varphi(s) = \varphi(s) \varphi(s^{-1}s)\varphi(s)^*=\varphi(s)\varphi(s)^*=\varphi(s).
			$$
			Replacing $s$ by $s^{-1}$, we obtain $\varphi(s^{-1})\varphi(s)\varphi(s^{-1}) = \varphi(s^{-1})$. It follows that $\varphi(s^{-1})=\varphi(s)^{-1}$, by the uniqueness of the inverse element.
			
			Similarly, one proves the implications (Inv) $\Rightarrow$ (LSl$'$) $\Rightarrow$  (LSl$''$) $\Rightarrow$ (Inv).  In view of Lemma~\ref{prop:inverse}, this completes the proof.
		\end{proof}
		
		\begin{remark} Let $S$ be a reduced restriction semigroup. Since the order on $S$ is trivial, a premorphism $\varphi\colon S\to T$ is locally strong if and only if it is strong.
		\end{remark}

		\begin{proposition} \label{prop:prop_strong}
			Let $\varphi\colon S\to T$ be a map satisfying (PM1), (LSl\,$'$\!) and (LSr\,$'$\!)  where $S$ is a restriction semigroup and $T$ an inverse semigroup. Then $\varphi$ is a premorphism.
		\end{proposition}

		\begin{proof} Let $e\in P(S)$. We first show that $\varphi(e)\in E(T)$.
			Since $e^*=e^+=e$, we have $\varphi(e)\varphi(e) = \varphi(e^2)\varphi(e)^* = \varphi(e)\varphi(e)^* = \varphi(e)$, so that $\varphi(e)\in E(T)$. 
			
			Let $s\in S$. Then $\varphi(s)\varphi(s^*) = \varphi(s)^+\varphi(ss^*) = \varphi(s)^+\varphi(s) = \varphi(s)$. It follows that
			$(\varphi(s)\varphi(s^*))^* = \varphi(s)^*$, that is, in view of \eqref{eq:rule1}, $\varphi(s)^*\varphi(s^*) = \varphi(s)^*$. Hence $\varphi(s)^*\leq \varphi(s^*)$. We have proved (PM2). Condition (PM3) is proved dually.
		\end{proof}
		
		Therefore, restricting one's attention to locally strong premorphisms from a restriction semigroup to an inverse semigroup, conditions (PM2) and (PM3) in the definition of a premorphism get redundant.
		
		The following result extends Theorem \ref{th:strong_order}.
		
		\begin{theorem}\label{th:strong_restr}
			Let $\varphi\colon S\to T$ be a premorphism between restriction semigroups. Then:
			\begin{enumerate}[(1)]
				\item $\varphi$ satisfies condition (Sr) if and only if it satisfies conditions (LSr) and (OP);
				\item $\varphi$ satisfies condition (Sl) if and only if it satisfies conditions (LSl) and (OP).
			\end{enumerate}
			Consequently, $\varphi$ is strong if and only if it is locally strong and  order-preserving.
			
		\end{theorem}
		
		\begin{proof} (1) We assume that $\varphi$  satisfies (Sr). Then it clearly satisfies (LSr) and we show that it satisfies (OP). Let $s\leq t$ in $S$. Then $s=ts^*$. We have:
			\begin{equation}\label{eq:mar6a}
			\varphi(t)\varphi(s^*) = \varphi(ts^*)\varphi(s^*)^* = \varphi(s)\varphi(s^*)^*.
			\end{equation}
			By  (PM2)  we have $\varphi(s^*)\varphi(s)^* = \varphi(s)^*$. It follows that
			\begin{equation}\label{eq:mar6b}
			\varphi(s)\varphi(s^*)^* = \varphi(s)\varphi(s^*) = \varphi(s)\varphi(s)^*\varphi(s^*) = \varphi(s)\varphi(s)^* = \varphi(s).
			\end{equation}
			From \eqref{eq:mar6a} and \eqref{eq:mar6b} we have $\varphi(t)\varphi(s^*) = \varphi(s)$ whence
			$$
			\varphi(t)\varphi(s)^* = \varphi(t)\varphi(s^*)\varphi(s)^* = \varphi(s)\varphi(s)^* = \varphi(s),
			$$ so that $\varphi(s)\leq \varphi(t)$.

			In the reverse direction, we assume that $\varphi$ satisfies (LSr) and (OP). For $s,t\in S$ we have:
			$$
			\varphi(s)\varphi(t) = \varphi(s)\varphi(t)\varphi(t)^* \leq \varphi(st)\varphi(t)^* = \varphi(st^+)\varphi(t) \leq \varphi(s)\varphi(t).
			$$
			It follows that $\varphi(s)\varphi(t) = \varphi(st)\varphi(t)^*$ so that condition (Sr) holds. 
			Part (2) is proved dually.
		\end{proof}
		\subsection{Locally multiplicative premorphisms}  Let $\varphi\colon S\to T$ be a premorphism between restriction semigroups and consider the following conditions:
		\begin{enumerate}
			\item[(M)] $\varphi(s)\varphi(t)=\varphi(st)$ for all $s,t\in S$;
			\item[(LM)] $\varphi(st^+)\varphi(s^*t) = \varphi(st)$ for all $s,t\in S$;
			\item[(LM$'$)] $\varphi(s)\varphi(t)=\varphi(st)$ for all $s,t\in S$ satisfying $s^*=t^+$.
		\end{enumerate}
		
		It is easily seen that conditions (LM) and (LM$'$) are equivalent, and that (M) implies~(LM).
		
		\begin{definition}\label{def:local_morphism} A premorphism $\varphi\colon S\to T$ between restriction semigroups is  called:
			\begin{itemize}
				\item {\em multiplicative} if it satisfies condition (M);
				\item {\em locally multiplicative} if it satisfies condition (LM).
			\end{itemize}
		\end{definition}

		Further, we consider the following conditions:
		\begin{enumerate}
			\item[(LMr)] $\varphi(st^+)\varphi(t)=\varphi(st)$ for all $s,t\in S$;
			\item[(LMl)] $\varphi(s)\varphi(s^*t) = \varphi(st)$ for all $s,t\in S$.
		\end{enumerate}

		\begin{proposition} \label{prop:local2} Let $\varphi\colon S\to T$ be a premorphism between restriction semigroups. The following statements are equivalent:
			\begin{enumerate}[(1)]
				\item $\varphi$ satisfies condition (M);
				\item $\varphi$ satisfies conditions (LM) and (OP);
				\item $\varphi$ satisfies conditions (LMr) and (OP);
				\item $\varphi$ satisfies conditions (LMl) and (OP).
			\end{enumerate}
		\end{proposition}
		
		\begin{proof} We prove implications (1) $\Rightarrow$ (3) $\Rightarrow$ (1).
			
			(1) $\Rightarrow$ (3) If $\varphi$ satisfies condition (M), it clearly satisfies (LMr). We show that it satisfies (OP). Let $s\leq t$ in $S$. Then $s=ts^*$ whence
			$\varphi(s) = \varphi(t)\varphi(s^*) \leq \varphi(t)$, as needed.
			
			(3) $\Rightarrow$ (1) 
			We assume that $\varphi$ satisfies conditions (LMr) and (OP) and let $s,t\in S$.
			Then $\varphi(st) = \varphi(st^+)\varphi(t) \leq \varphi(s)\varphi(t) \leq \varphi(st)$. It follows that $\varphi(st) = \varphi(s)\varphi(t)$, as needed.
			
			Implications (1) $\Rightarrow$ (2) $\Rightarrow$ (1) and (1) $\Rightarrow$ (4) $\Rightarrow$ (1) are proved similarly.
		\end{proof}
		
		In Subsection \ref{subs:perfect} we will show that proper extensions related to multiplicative and locally multiplicative premorphisms generalize the almost perfect restriction semigroups from \cite{Jones16, Kud15}.
		
		\section{Structure of proper extensions of restriction semigroups}\label{s:structure}
		
		In this section we prove a structure result for proper extension of restriction semigroups in terms of partial actions that extends Cornock's and Gould's result on the structure of proper restriction semigroups \cite{CG} as well as O'Carroll's result on the structure of idempotent pure extensions of inverse semigroups \cite{OC77},  or, more precisely,  its partial action variant obtained by the second author in \cite{Kh17}.
		
		\subsection{Proper morphisms and proper extensions of restriction semigroups.}
		
		A morphism $\psi\colon S\to T$ between restriction semigroups is called {\em proper} if it is surjective and
		$$
		\psi(s)=\psi(t) \Rightarrow s\sim t, \, {\text{ for all}} \, s,t\in S.
		$$
		
		Recall that the generalized Green's relations ${\widetilde{\mathcal R}}$ and ${\widetilde{\mathcal L}}$ on a restriction semigroup $S$ are defined by $a \mathrel{\widetilde{\mathcal R}} b$ (respectively $a \mathrel{\widetilde{\mathcal L}} b$) if and only if $a^+=b^+$ (respectively $a^*=b^*$). 
		
		\begin{proposition} \label{prop:proper} A surjective morphism $\psi\colon S\to T$ between restriction semigroups is proper if and only if the restriction of $\psi$ to ${\widetilde{\mathcal R}}$- and ${\widetilde{\mathcal L}}$-classes is injective.
		\end{proposition}

		\begin{proof} Assume that $\psi$ is proper, $s \mathrel{{\widetilde{\mathcal R}}} t$ and $\psi(s)=\psi(t)$. Then $s\sim t$ and thus
			$s=s^+s=t^+s = s^+t=t$, so that the restriction of $\psi$ to an ${\widetilde{\mathcal R}}$-class is injective. The case of an ${\widetilde{\mathcal L}}$-class is dual.
			
			Suppose now that the restriction of $\psi$ to ${\widetilde{\mathcal R}}$- and ${\widetilde{\mathcal L}}$-classes is injective and let $\psi(s)=\psi(t)$. Then $\psi(s^+t) = \psi(t^+s)$. Since $s^+t \mathrel{{\widetilde{\mathcal R}}} t^+s$ it follows that $s^+t=t^+s$. Dually one proves that $st^*=ts^*$. Therefore, $s\sim t$, and $\psi$ is proper.
		\end{proof}
		
		\begin{corollary}
		If $T$ is a reduced restriction semigroup and $S$ is a proper restriction semigroup then a surjective morphism $\psi\colon S\to T$ is proper if and only if  ${\mathrm{ker}}(\psi)=\sigma$, in particular, $T\simeq S/\sigma$. 
		\end{corollary}
		\begin{proof} Suppose that $\psi\colon S\to T$ is proper. Since $\sigma$ is the minimum reduced congruence on $S$, we have $\sigma \subseteq  {\mathrm{ker}}(\psi)$.  Let $\psi(a) = \psi(b)$. Then $\psi(ab^*) =\psi(ba^*)$. Since also $ab^* \mathrel{{\widetilde{\mathcal L}}} ba^*$ it follows that $ab^*=ba^*$. Therefore $a \mathrel{\sigma} b$ and thus $\sigma =  {\mathrm{ker}}(\psi)$. On the other hand, because $\sigma = \, \sim$, the canonical projection morphism induced by $\sigma$ is proper. \end{proof}
		
		It follows that proper extensions of restriction semigroups generalize proper restriction semigroups.
		It is known and easily seen that a morphism between inverse semigroups is proper if and only if it is idempotent pure. Hence proper morphisms between restriction semigroups also generalize idempotent pure morphisms between inverse semigroups.
		
	For restriction semigroups, one can define a morphism $\psi\colon S\to T$ to be {\em projection pure} if $\varphi^{-1}(e) \subseteq P(S)$ whenever $e\in P(T)$. A proper morphism is necessarily projection pure, but the converse does not hold in general. It is a long established idea that the notion of projection pure morphisms is not the appropriate generalization of the notion of idempotent-pure morphisms between inverse semigroups. For example, the statement of Proposition \ref{prop:proper} does not hold if `proper' in its formulation is replaced by `projection pure' because, e.g., a projection pure morphism from a reduced restriction semigroup does not need to be injective (since a congruence on a monoid is not in general uniquely determined by the class of $1$).
		
		For further use, we record the following characterization of a proper morphism.
		
		\begin{lemma} \label{lem:proper_compatibility} Let $\psi\colon S\to T$ be a surjective morphism between restriction semigroups. Then $\psi$ is proper if and only if 
			$$s\sim t \, \Longleftrightarrow \, \psi(s)\sim \psi(t), \, {\text{ for all }} \, s,t\in S.$$
		\end{lemma}
		
		\begin{proof} If $s\sim t$ then clearly $\psi(s)\sim \psi(t)$. We assume that $\psi(s)\sim \psi(t)$ and show that $s\sim t$. Indeed, $\psi(s)\psi(t)^* = \psi(t)\psi(s)^*$ which implies that
			$\psi(st^*) = \psi(ts^*)$. Because $\psi$ is proper this implies that $st^* \sim ts^*$. Therefore, in view of $(st^*)^* = (ts^*)^*$, we obtain $st^*=ts^*$ (applying Lemma \ref{lem:aux14}(1)). Similarly one shows that $t^+s=s^+t$.
			Hence $s\sim t$, as desired.
		\end{proof}
		
		An analogous notion of a proper extension for weakly left ample semigroups has been considered in \cite{Gomes05} and for a class of generalized left restriction semigroups ({\em glrac} semigroups) in \cite{BGG10}.
		
		\subsection{The partial action product of $Y$ by $S$}\label{subs:construction}
		Let $S$ be a restriction semigroup, $Y$ a semilattice, $q\colon Y\to P(S)$ a  morphism of semilattices and $\varphi\colon S\to {\mathcal{I}}(Y)$ a premorphism. We introduce the following conditions involving the triple $(\varphi, q, Y)$:
		\begin{enumerate}
			\item[(A1)]  for all $s\in S$: $\operatorname{\mathrm{dom}}\varphi_s$ and  $\operatorname{\mathrm{ran}}\varphi_s$ are order ideals;
			\item[(A2)] for all $s\in S$: $\varphi_s$ is an order isomorphism;
			\item[(A3)]  for all $e\in P(S)$:  $(q^{-1}(e))^{\downarrow} \subseteq \operatorname{\mathrm{dom}}{\varphi_e} \subseteq q^{-1}(e^{\downarrow})$;
			\item[(A4)] for all $s\in S$: $\operatorname{\mathrm{dom}}\varphi_s\cap \{y\in Y\colon q(y)= s^*\}\neq \varnothing$.
		\end{enumerate}
		
		Observe that (A1) and (A2) tell us that the image of $\varphi$ belongs to the semigroup $\Sigma(Y)$ of all order isomorphisms between order ideals of $Y$ \cite[VI.7.1]{Petrich}.
		
				Due to (PM2) and (PM3), the following condition is satisfied:
		
		\begin{enumerate}
			\item[(A5)] for all $s\in S$: $\operatorname{\mathrm{dom}}\varphi_s\subseteq \operatorname{\mathrm{dom}}\varphi_{s^*}$ and $\operatorname{\mathrm{ran}}\varphi_s\subseteq \operatorname{\mathrm{ran}}\varphi_{s^+}$.
		\end{enumerate}
		
		We also specify the following two conditions that will arise in Subsection \ref{subs:underlying}:
		
		\begin{enumerate}
			\item[(A3a)]   For all $s\in S$:  $\operatorname{\mathrm{dom}}{\varphi_s}= \{y\in Y\colon y\in \operatorname{\mathrm{dom}}{\varphi_t}\cap q^{-1}(t^*) \text{ for some } t\leq s\}$;
			\item[(A3b)]   For all $s\in S$:  $\operatorname{\mathrm{dom}}{\varphi_s}= (\operatorname{\mathrm{dom}}{\varphi_s}\cap q^{-1}(s^*))^{\downarrow}$.
		\end{enumerate}

		Suppose that the triple $(\varphi, q, Y)$ satisfies conditions (A1)--(A4). It follows from condition (A3) that $q^{-1}(e)\subseteq \operatorname{\mathrm{dom}}{\varphi_e}$ for all $e\in P(S)$. Condition (A4) implies that $q$ is surjective. If, in addition, (A3a) holds, then, for any $s,t\in S$ such that $s\leq t$ we have $\operatorname{\mathrm{dom}}{\varphi_s} \subseteq \operatorname{\mathrm{dom}}{\varphi_t}$. This, in view of Lemma \ref{lem:prem_compatibility}  and Lemma \ref{lem:lem1}\eqref{i:b7}, implies that $\varphi_s \leq \varphi_t$ whence $\varphi$ is order-preserving. The remaining claims of the following are easy to prove. 
		
		\begin{lemma} \label{lem:prop_hat_tilde1} Suppose that the triple $(\varphi, q, Y)$ satisfies conditions (A1)--(A4). 
			\begin{enumerate}[(1)]
				\item If, in addition, $(\varphi, q, Y)$ satisfies (A3a), then for all $e\in P(S)$: $\operatorname{\mathrm{dom}}\varphi_e = q^{-1} (e^{\downarrow})$ and, moreover, $\varphi$ is order-preserving.
				\item If, in addition, $(\varphi, q, Y)$ satisfies (A3b), then for all $e\in P(S)$: $\operatorname{\mathrm{dom}}{\varphi_e}= (q^{-1}(e))^{\downarrow}$.
			\end{enumerate} 
		\end{lemma}
		
		We will  need the following observation.
		
		\begin{lemma}\label{lem:q} Suppose that the triple $(\varphi, q, Y)$ satisfies conditions (A1)--(A4), $s\in S$ and $y\in \operatorname{\mathrm{dom}}\varphi_s$. Then:
			\begin{enumerate}[(1)]
				\item\label{i:lem:q1}  $q(\varphi_s(y)) = (sq(y))^+$;
				\item\label{i:lem:q2} $q(y) = (q(\varphi_s(y))s)^*$.
			\end{enumerate}
		\end{lemma}
		
		\begin{proof}  Let $y\in \operatorname{\mathrm{dom}}\varphi_s$. Put $e=q(y)$ and $f=q(\varphi_s(y))$. Because $y\in \operatorname{\mathrm{dom}}\varphi_e$ with $\varphi_e(y) = y$ (due to (A3)) and
			$\varphi_s\varphi_e\leq \varphi_{se}$, we have that $y\in \operatorname{\mathrm{dom}}\varphi_{se}$ and thus $\varphi_{se}(y) = \varphi_s(y)$. It follows that $\varphi_s(y)\in \operatorname{\mathrm{ran}}\varphi_{se}$, whence $\varphi_s(y)\in \operatorname{\mathrm{ran}}\varphi_{(se)^+}$, by (A5). Due to (A3) this yields 
			\begin{equation}\label{eq:aux11}
			f\leq (se)^+.
			\end{equation}
			
			Similarly one shows that 
			\begin{equation}\label{eq:aux12}
			e\leq (fs)^*.
			\end{equation}
			
			Then 
			\begin{align*}
			fs & \leq (se)^+s = se & (\text{by } \eqref{eq:aux11} \text{ and } \eqref{eq:mov_proj})\\
			&  \leq s(fs)^* = fs & (\text{by } \eqref{eq:aux12} \text{ and } \eqref{eq:mov_proj})
			\end{align*}
			and therefore $fs=se$. It follows that 
			\begin{align*}
			f & = fs^+  & (\text{by } \eqref{eq:aux11} \text{ and since } (se)^+\leq s^+)\\
			&  = (fs)^+=(se)^+ & (\text{by } \eqref{eq:aux12} \text{ and since } fs=se)
			\end{align*}
			and similarly $e=(fs)^*$, which completes the proof.
			\end{proof}
		
		Suppose that the triple $(\varphi, q, Y)$ satisfies conditions (A1)--(A4). The {\em partial action product} $Y\rtimes^q_{\varphi} S$ of $Y$ by $S$ with respect to $\varphi$ and $q$ is the  algebra
		$(Y\rtimes^q_{\varphi} S; \cdot\, , ^*, +)$ of signature $(2,1,1)$ defined as follows. Its underlying set is
		$$
		Y\rtimes^q_{\varphi} S = \{(y,s)\in Y\times S \colon  y\in \operatorname{\mathrm{ran}}\varphi_s \,\, {\text{and}} \,\, q(y)= s^+\};
		$$
		the multiplication (denoted by juxtaposition) is given by
		\begin{align}\label{product-in-Y-rt-S}
		(y,s)(x,t) = (\varphi_s(\varphi_s^{-1}(y)\wedge x), st);
		\end{align}
		and the unary operations $^*$ and $^+$ are given by
		\begin{equation} \label{eq:def_star_plus}
		(y,s)^* = (\varphi_s^{-1}(y), s^*),  \,\, (y,s)^+ = (y,s^+). 
		\end{equation}
		
		Note first that $Y\rtimes^q_{\varphi} S\neq \varnothing$. Indeed, let $s\in S$. By (A4) there is $y\in \operatorname{\mathrm{dom}}\varphi_s \cap q^{-1}(s^*)$. Lemma \ref{lem:q}(1) now implies that $\varphi_s(y)\in \operatorname{\mathrm{ran}}\varphi_s \cap q^{-1}(s^+)$ so that $(\varphi_s(y),s)\in Y\rtimes^q_{\varphi} S$.
		
		\begin{lemma} The operations on $Y\rtimes^q_{\varphi} S$ are well defined.
		\end{lemma}
		\begin{proof} We show that the multiplication is well defined. Let $(y,s),(x,t)\in Y\rtimes^q_{\varphi} S$. Note that $\varphi_s^{-1}(y)\wedge x \leq x\in \operatorname{\mathrm{ran}}\varphi_t$ by (A1). Since $\varphi_s^{-1}(y) \in \operatorname{\mathrm{dom}}\varphi_s$, we have $\varphi_s^{-1}(y)\wedge x \in \operatorname{\mathrm{dom}}\varphi_s$ by (A1). It follows that  $$
			\varphi_s(\varphi_s^{-1}(y)\wedge x) \in \varphi_s(\operatorname{\mathrm{ran}}\varphi_t \cap \operatorname{\mathrm{dom}}\varphi_s) = \operatorname{\mathrm{ran}}\varphi_s\varphi_t \subseteq \operatorname{\mathrm{ran}}\varphi_{st}.
			$$
			We are left to show that $q(\varphi_s(\varphi_s^{-1}(y)\wedge x))=(st)^+$. We calculate:
			\begin{align*}
			q(\varphi_s(\varphi_s^{-1}(y)&\wedge x)) = (sq(\varphi_s^{-1}(y)\wedge x))^+ & (\text{by  Lemma }\ref{lem:q}(1))\\
			& = (sq(\varphi_s^{-1}(y))q(x))^+ & (\text{since } q \text{ is a morphism})\\
			& = (s(q(y)s)^*q(x))^+ & (\text{by  Lemma }\ref{lem:q}(2))\\
			& = (q(y)sq(x))^+ = q(y)(sq(x))^+ & (\text{by } \eqref{eq:mov_proj} \text{ and } \eqref{eq:rule1})\\
			& = s^+(st^+)^+ & (\text{since } q(y)=s^+ \text{ and } q(x)=t^+)\\
			& = (st)^+, & (\text{since } (st^+)^+\leq s^+ \text{  and by } \eqref{eq:consequences} )
			\end{align*}
			as desired.
			
			We now show that $(y,s)^*$ is well defined. From Lemma \ref{lem:q}(2) we have $q(\varphi_s^{-1}(y))=(q(y)s)^*$. In view of $q(y)=s^+$, this yields that $q(\varphi_s^{-1}(y))=s^*= (s^*)^+$. Since $y \in \operatorname{\mathrm{ran}}\varphi_s$, it follows from (A5) that $\varphi_s^{-1}(y) \in \operatorname{\mathrm{dom}}\varphi_{s^*}=\operatorname{\mathrm{ran}}\varphi_{s^*}$. Hence $(y,s)^*\in Y\rtimes^q_{\varphi} S$. That $(y,s)^+\in Y\rtimes^q_{\varphi} S$ is even easier to prove.
		\end{proof}
		
		\begin{proposition} $(Y\rtimes^q_{\varphi} S; \cdot\, , ^*, +)$ is a restriction semigroup.
		\end{proposition}
		
		\begin{proof} 
			
			The proof amounts to a routine verification of the axioms of a restriction semigroup. We first prove that the multiplication is associative. 
			Let $(x,s),(y,t),(z,u)\in Y\rtimes^q_{\varphi} S$. We need to show that
			\begin{equation}\label{eq:assoc0}
			\bigl((x,s)(y,t)\bigr)(z,u) = (x,s) \bigl((y,t)(z,u)\bigr).
			\end{equation}
			We calculate each of these two products and then show that the results coincide.
			We put $x'=\varphi_s^{-1}(x)$ and $y'=\varphi_t^{-1}(y)$. The left-hand side of \eqref{eq:assoc0} is equal to
			\begin{equation}\label{eq:assoc1}
			(\varphi_s(x'\wedge y),st)(z,u) = 
			(\varphi_{st}\bigl(\varphi_{st}^{-1}(\varphi_s(x'\wedge y))\wedge z\bigr), stu). 
			\end{equation}
			Because $x'\wedge y \leq y \in \operatorname{\mathrm{ran}}\varphi_t$ and $\operatorname{\mathrm{ran}}\varphi_t$ is an order ideal,  $x'\wedge y \in \operatorname{\mathrm{ran}}\varphi_t$.
			Let 
			$$w=\varphi_t^{-1}(x'\wedge y).
			$$
			Then $\varphi_s(x'\wedge y) = \varphi_s\varphi_t(w)$.
			Since $\varphi_s\varphi_t\leq \varphi_{st}$ it follows that $w\in \operatorname{\mathrm{dom}}\varphi_{st}$ and also $\varphi_{st}(w)=\varphi_s\varphi_t(w)$. Hence $\varphi_s(x'\wedge y) = \varphi_{st}(w)$. Further, $\varphi_{st}^{-1}\varphi_{st}(w) = w$, thus \eqref{eq:assoc1} is equivalent to
			\begin{equation*}\label{eq:assoc2}
			\left((x,s)(y,t)\right)(z,u) = (\varphi_{st}(w\wedge z), stu).
			\end{equation*}
			
			The right-hand side of \eqref{eq:assoc0} is equal to 
			\begin{equation*}\label{eq:assoc3}
			(x,s) (\varphi_t(y'\wedge z),tu) = 
			(\varphi_s\bigl(x'\wedge \varphi_t(y'\wedge z)\bigr), stu).
			\end{equation*}
			We thus need to show that
			\begin{equation}\label{eq:assoc4}
			\varphi_{st}(w\wedge z)=\varphi_s\bigl(x'\wedge \varphi_t(y'\wedge z)\bigr).
			\end{equation}
			Since $w\in \operatorname{\mathrm{dom}} \varphi_s\varphi_t$, we have  $w \wedge z\in \operatorname{\mathrm{dom}} \varphi_s\varphi_t$ and
			\mbox{$\varphi_{st}(w\wedge z)=\varphi_s\varphi_t(w\wedge z)$}. Applying $\varphi_s^{-1}$ to both sides of \eqref{eq:assoc4}, we obtain the equivalent equality
			\begin{equation}\label{eq:assoc5}
			\varphi_{t}\bigl(w\wedge z\bigr)=x'\wedge \varphi_t(y'\wedge z).
			\end{equation}
			
			Since $\varphi_t(w)= x'\wedge y\leq y$, it follows that $w\leq y'$. Hence the left-hand side of \eqref{eq:assoc5} is equal to  $\varphi_{t}\bigl(w\wedge y'\wedge z\bigr)$.
			Because $\varphi_t$ is an order automorphism, it respects the operation $\wedge$, i.e., $\varphi_t(a\wedge b)=\varphi_t(a)\wedge \varphi_t(b)$ whenever $a,b\in \operatorname{\mathrm{dom}}\varphi_t$. It follows that the left-hand side of \eqref{eq:assoc5} is equal to
			\begin{equation*}
			\varphi_{t}\bigl(w\wedge y'\wedge z\bigr) = \varphi_t(w) \wedge \varphi_t(y'\wedge z) = x'\wedge y \wedge \varphi_t(y'\wedge z).\\
			\end{equation*}
			Since $y'\wedge z\leq y'$, the latest expression is equal to $x' \wedge  \varphi_t(y'\wedge z)$, which is precisely the right-hand side of \eqref{eq:assoc5}. The associativity of the multiplication in $Y\rtimes^q_{\varphi} S$ is established.
			
			We now verify axioms \eqref{eq:axioms:plus}. Let $(y,s),(z,t)\in Y\rtimes^q_{\varphi} S$. We have:
			\begin{align*}
			(y,s)^+(y,s)&=(\varphi_{s^+}(y),s^+)(y,s)   & (\text{by (A5) and \eqref{eq:def_star_plus}})\\
			&=(\varphi_{s^+}(y),s^+s)              & (\text{by \eqref{product-in-Y-rt-S}})\\
			&=(y,s),                               & (\text{by (A5) and since } s^+s=s)
			\end{align*}
			so that the first axiom in \eqref{eq:axioms:plus} is verified.
			For the second axiom we calculate:
			\begin{align*}
			(y,s)^+(z,t)^+&=(\varphi_{s^+}(y),s^+)(z,t^+)   & (\text{by (A5) and \eqref{eq:def_star_plus}})\\
			&=(\varphi_{s^+}(y\wedge z),s^+t^+)                & \text{(by \eqref{product-in-Y-rt-S})}\\
			&=(y\wedge z,s^+t^+).                         & \text{(by (A5))}
			\end{align*}
			By symmetry, we also have the equality $(z,t)^+(y,s)^+=(z\wedge y,t^+s^+)$. It follows that
			$(y,s)^+(z,t)^+=(z,t)^+(y,s)^+$, as needed.
			For the third axiom in \eqref{eq:axioms:plus}, we calculate:
			\begin{align*}
			((y,s)^+(z,t))^+&=((y,s^+)(z,t))^+    & \text{(by \eqref{eq:def_star_plus})}\\
			&=(\varphi_{s^+}(y\wedge z),s^+t)^+           & \text{(by \eqref{product-in-Y-rt-S})}\\
			&=(y\wedge z,(s^+t)^+) = (y\wedge z,s^+t^+).                    & \text{(by \eqref{eq:def_star_plus} and \eqref{eq:axioms:plus})}
			\end{align*}
			But we have verified above that  $(y,s)^+(z,t)^+=(y\wedge z,s^+t^+)$. We obtain the equality
			$((y,s)^+(z,t))^+=(y,s)^+(z,t)^+$.
			We finally verify the fourth axiom in \eqref{eq:axioms:plus}. We put $y'=\varphi_s^{-1}(y)$. Then
			\begin{align*}
			((y,s)(z,t))^+(y,s)&=(\varphi_s(y'\wedge z),(st)^+)(y,s)  & \text{(by  \eqref{product-in-Y-rt-S} and \eqref{eq:def_star_plus})}\\
			&=(\varphi_{(st)^+}(\varphi_s(y'\wedge z)\wedge y),(st)^+s)     & \text{(by \eqref{product-in-Y-rt-S})}\\
			&=(\varphi_s(y'\wedge z)\wedge y,(st)^+s)               & \text{(by Lemma \ref{lem:proj_identity})}\\
			&=(\varphi_s(y'\wedge z),st^+)                          & \text{(applying $\varphi_s(y'\wedge z)\le y$  and  \eqref{eq:axioms:plus})}\\
			&=(y,s)(z,t^+)=   (y,s)(z,t)^+.                                & \text{(by \eqref{product-in-Y-rt-S} and \eqref{eq:def_star_plus})}
			\end{align*}
			We now turn to checking axioms \eqref{eq:axioms:star}. 
			We put $y'=\varphi_s^{-1}(y)$ and $z'=\varphi_t^{-1}(z)$. For the first axiom in  \eqref{eq:axioms:star}, we calculate:
			\begin{align*}
			(y,s)(y,s)^*&=(y,s)(y', s^*)   & \text{(by \eqref{eq:def_star_plus})}\\
			&=(\varphi_s(\varphi_s^{-1}(y)\wedge y'), ss^*)     & \text{(by \eqref{product-in-Y-rt-S})}\\
			&=(y,s).                       & 
			\end{align*}
			For the second axiom in  \eqref{eq:axioms:star}, we calculate:
			\begin{align*}
			(y,s)^*(z,t)^*&=(y', s^*)(z', t^*)   & \text{(by  \eqref{eq:def_star_plus})}\\
			&=(y'\wedge z', s^*t^*).                & \text{(by \eqref{eq:def_star_plus} and (A5))}
			\end{align*}  
			By symmetry,  $(z,t)^*(y,s)^*=(z'\wedge y', t^*s^*)$.
			Hence $(y,s)^*(z,t)^*=(z,t)^*(y,s)^*$.
			The third axiom in \eqref{eq:axioms:star} holds because
			\begin{align*}
			((y,s)(z,t)^*)^*&=((y,s)(z', t^*))^*        & \text{(by \eqref{eq:def_star_plus})}\\
			&=(\varphi_s(y'\wedge z'), st^*)^*                & \text{(by \eqref{product-in-Y-rt-S})}\\
			&=(y'\wedge z',s^*t^*), & \text{(by \eqref{eq:def_star_plus} and \eqref{eq:axioms:star})}
			\end{align*}
			which, as shown above, is equal to $(y,s)^*(z,t)^*$.
			For the fourth axiom in \eqref{eq:axioms:star}, we first observe that
			\begin{equation}\label{eq:axiom4}
			(y,s)((z,t)(y,s))^* =(y,s)(\varphi_t(z'\wedge y),ts)^* = (y,s)(\varphi_{ts}^{-1}\varphi_t(z'\wedge y),(ts)^*).
			\end{equation}
			Since $z'\wedge y\leq y$, we have that $z'\wedge y\in \operatorname{\mathrm{ran}}\varphi_s$ whence 
			$z'\wedge y= \varphi_s\varphi_s^{-1}(z'\wedge y)$. Because, in addition, $\varphi_t\varphi_s\leq \varphi_{ts}$,
			the right-hand side of \eqref{eq:axiom4} can be rewritten as 
			$$(y,s)(\varphi_{ts}^{-1}\varphi_{ts}\varphi_s^{-1}(z'\wedge y),(ts)^*) = (y,s)(\varphi_s^{-1}(z'\wedge y),(ts)^*),
			$$
			which, by the definition of the multiplication \eqref{product-in-Y-rt-S} and \eqref{eq:axioms:star}, is equal to 
			$$
			(\varphi_s(\varphi_s^{-1}(y)\wedge \varphi_s^{-1}(z'\wedge y)), s(ts)^*) = (\varphi_s\varphi_s^{-1}(z'\wedge y), t^*s)= (z'\wedge y, t^*s).
			$$
			It remains to observe that
			$$
			(z,t)^*(y,s) = (z',t^*)(y,s) = (\varphi_{t^*}(z'\wedge y),t^*s) = (z'\wedge y, t^*s).
			$$
			
			Axioms \eqref{eq:axioms:common} that relate $^+$ with $^*$ are easy to check: for $(y,s)\in Y\rtimes^q_{\varphi} S$ we have
			$$
			((y,s)^+)^* =(y,s^+)^*  = (y,s^+) = (y,s)^+ 
			$$
			and, similarly, $((y,s)^*)^+ = (y,s)^*$. This completes the proof.
		\end{proof}

		In the following lemma we collect some properties of $Y\rtimes^q_{\varphi} S$.
		
		\begin{lemma}\label{lem:properties_extension}\mbox{}
			\begin{enumerate}[(1)]
				\item The semilattice $P(Y\rtimes^q_{\varphi} S)$ is isomorphic to the semilattice $Y$ via the map $(y,q(y))\mapsto y$.
				\item Let $(y,s),(z,t)\in Y\rtimes^q_{\varphi} S$. Then
				$(y,s)\leq (z,t)$ if and only if $y\leq z$ and $s\leq t$.
				\item  Let $(y,s),(z,t)\in Y\rtimes^q_{\varphi} S$. Then
				$(y,s)\sim(z,t)$ if and only if $s\sim t$.
			\end{enumerate}
		\end{lemma}
		
		\begin{proof}
			(1) Let $(y,s) \in P(Y\rtimes^q_{\varphi} S)$. Then $(y,s) = (y,s)^+ = (y, s^+) = (y, q(y))$. On the other hand, (A3) implies that $(y, q(y)) \in Y\rtimes^q_{\varphi} S$ whence $(y, q(y)) \in P(Y\rtimes^q_{\varphi} S)$ for any $y\in Y$. The claim now easily follows.
			
			(2) $(y,s)\leq (z,t)$ holds if and only if $(y,s) = (y,s)^+(z,t) = (y,s^+)(z,t) = (y\wedge z, s^+t)$. Thus $(y,s)\leq (z,t)$ is equivalent to $y\leq z$ and $s\leq t$.
			
			(3) Observe that 
			$$(y,s)^+(z,t)=(y,s^+)(z,t) =(\varphi_{s^+}(y\wedge z),s^+t) = (y\wedge z,s^+t),$$
			thus also $(z,t)^+(y,s)=(z\wedge y,t^+s)$.  
			It follows that $(y,s)^+(z,t)=(z,t)^+(y,s)$ holds if and only if $s^+t=t^+s$.
			
			Let $y'=\varphi_s^{-1}(y)$ and $z'=\varphi_t^{-1}(z)$. We have
			$$(y,s)(z,t)^*=(y,s)(z',t^*)=(\varphi_s(y'\wedge z'),st^*).$$
			Since $z'\in \operatorname{\mathrm{dom}}\varphi_t$ it follows from (A5) that $z'\in \operatorname{\mathrm{dom}}\varphi_{t^*}$. Hence
			$y'\wedge z' \in \operatorname{\mathrm{dom}}(\varphi_{t^*})$ by (A2). 
			It follows that $\varphi_s(y'\wedge z')=\varphi_s\varphi_{t^*}(y'\wedge z')$. Because $\varphi_s\varphi_{t^*}\leq \varphi_{st^*}$, we have that
			$y'\wedge z'\in \mathrm{dom}\varphi_{st^*}$ and $\varphi_s\varphi_{t^*}(y'\wedge z')=\varphi_{st^*}(y'\wedge z')$.
			Hence $(y,s)(z,t)^* = (\varphi_{st^*}(y'\wedge z'),st^*)$. Then also $(z,t)(y,z)^* = (\varphi_{ts^*}(y'\wedge z'),ts^*)$.
			It follows that $(y,s)(z,t)^*=(z,t)(y,z)^*$ holds if and only if $st^*=ts^*$. This finishes the proof.
		\end{proof}
		
		We define the map  $\Psi\colon Y\rtimes^q_{\varphi} S \to S$ by
		\begin{equation}\label{eq:def_Psi}
		\Psi(y,s) = s, \,\, (y,s)\in Y\rtimes^q_{\varphi} S.
		\end{equation}
		
		It is immediate that $\Psi$ preserves the multiplication and the unary operations $^*$ and $^+$. Thus it is a morphism of restriction semigroups.
		Lemma \ref{lem:properties_extension}(3) implies that $\Psi$ is proper.
		
		\subsection{The two underlying premorphisms of a proper extension}\label{subs:underlying}
		Let $\psi\colon T\to S$ be a proper morphism between restriction semigroups. We put $p=\psi|_{P(T)}$. Then $p\colon P(T) \to P(S)$ is a morphism of semilattices.
		
		Let $s\in S$. We introduce partially defined maps $\widehat{\psi}_s \colon P(T)\to P(T)$ and $\widetilde{\psi}_s \colon P(T)\to P(T)$ by setting
		$$
		\operatorname{\mathrm{dom}}\widehat{\psi}_s = \{e\in P(T)\colon e\leq t^* \text{ for some } t\in T \text{ such that } \psi(t)\leq s\},
		$$
		$$
		\operatorname{\mathrm{dom}}\widetilde{\psi}_s = \{e\in P(T)\colon e\leq t^* \text{ for some } t\in T \text{ such that } \psi(t)= s\}.
		$$
		For $e\in \operatorname{\mathrm{dom}}\widehat{\psi}_s$ we set
		\begin{equation*}\label{eq:def_psi_hat}
		\widehat{\psi}_s(e)=(te)^+ \text{ where } t\in T \text{ is such that } e\leq t^* \text{ and } \psi(t)\leq s.
		\end{equation*}
		Similarly, for $e\in \operatorname{\mathrm{dom}}\widetilde{\psi}_s$ we set  
		\begin{equation*}
		\widetilde{\psi}_s(e)=(te)^+ \text{ where } t\in T \text{ is such that } e\leq t^* \text{ and } \psi(t)= s.
		\end{equation*}
		
		\begin{example} {\em In the case where $\psi\colon S\to S$ is the identity morphism, for each $s\in S$ we have $\widehat{\psi}_s = \widetilde{\psi}_s$ and the map $s\mapsto \widehat{\psi}_s$ coincides with the Munn representation of $S$.}
		\end{example}
		
		\begin{lemma} \label{lem:well_def}\mbox{} 
			\begin{enumerate}[(1)]
				\item For all $s\in S$ and $e\in  \operatorname{\mathrm{dom}}\widehat{\psi}_s$  the value $\widehat{\psi}_s(e)$  is well defined. 
				\item For all $s\in S$ and $e\in  \operatorname{\mathrm{dom}}\widetilde{\psi}_s$  the value $\widetilde{\psi}_s(e)$  is well defined. 
				\item \label{ij2:3} For all $s\in S$ we have $\operatorname{\mathrm{dom}}\widetilde{\psi}_s\subseteq \operatorname{\mathrm{dom}}\widehat{\psi}_s$ and $\widetilde{\psi}_s(e) = \widehat{\psi}_s(e)$ for all $e\in \operatorname{\mathrm{dom}}\widetilde{\psi}_s$. That is, $\widetilde{\psi}_s$ is the restriction of $\widehat{\psi}_s$ to the set $\operatorname{\mathrm{dom}}\widetilde{\psi}_s$. 
				\item For all $s\in S$ the maps $\widehat{\psi}_s$ and $\widetilde{\psi}_s$ are injective. 
				\item For all $e\in P(S)$ the maps $\widehat{\psi}_e$ and $\widetilde{\psi}_e$ are the identity maps on their domains.
			\end{enumerate}
		\end{lemma}
		
		\begin{proof} (1) Let $e\leq t^*, r^*$ where $\psi(t), \psi(r)\leq s$. Then $\psi(t)\sim \psi(r)$. By Lemma \ref{lem:proper_compatibility} we have $t\sim r$.
			It follows that $(te)^+ = (tr^*e)^+ = (rt^*e)^+ = (re)^+$. Hence $\widehat{\psi}_s(e)$ is well defined. 
			
			(2) Let $e\leq t^*, r^*$ where $\psi(t) = \psi(r) = s$. Then $t\sim r$, since $\psi$ is proper. The same calculation as in the proof of part (1) shows that 
			that $(te)^+ = (re)^+$. Hence $\widetilde{\psi}_s(e)$ is well defined. 

			(3) This is immediate by the definitions of $\widehat{\psi}_s$ and  $\widetilde{\psi}_s$.
			
			(4) Let $e,f \in  \operatorname{\mathrm{dom}}\widehat{\psi}_s$, where $e\leq t^*$, $f\leq r^*$ with $t,r \in \psi^{-1}({s^{\downarrow}})$ and  assume that $\widehat{\psi}_s(e) = \widehat{\psi}_s(f)$, that is, $(te)^+ = (rf)^+$. Because $\psi(t) \sim \psi(r)$,  Lemma \ref{lem:proper_compatibility} implies that $t\sim r$. Using Lemma \ref{lem:aux14}(2) we obtain $te \sim rf$, thus $te=rf$, by Lemma \ref{lem:lem1}\eqref{i:b8}.
			Hence $e = (te)^* = (rf)^* =f$, so that $\widehat{\psi}_s$ is injective.
			The injectivity of $\widetilde{\psi}_s$ now follows from~(3).
			
			(5) Let $x\in \operatorname{\mathrm{dom}}\widehat{\psi}_e$. Then $x\leq t^*$, where $\psi(t) \leq e$. Since $e\in P(S)$, we have $\psi(t) = e\psi(t)^* \in P(S)$.  Since $\psi$ is proper, it is projection pure, so that $t\in P(T)$ and thus $t=t^*$. We then have $\widehat{\psi}_e(x) = (tx)^+=(t^*x)^+ = x^+=x$, as needed. The claim about $\widetilde{\psi}_e$ follows from part (3).
		\end{proof}
		
		We have defined two maps $\widehat{\psi} \colon S\to {\mathcal I}(P(T))$, $s\mapsto \widehat{\psi}_s$, and $\widetilde{\psi} \colon S\to {\mathcal I}(P(T))$, $s\mapsto \widetilde{\psi}_s$.
		Observe that $\widetilde{\psi}_s\leq \widehat{\psi}_s$ for all $s\in S$. Furthermore,
		$$
		\operatorname{\mathrm{ran}}\widehat{\psi}_s = \{e\in P(T)\colon e\leq t^+ \text{ for some } t\in T \text{ such that } \psi(t)\leq s\},
		$$
		$$
		\operatorname{\mathrm{ran}}\widetilde{\psi}_s = \{e\in P(T)\colon e\leq  t^+ \text{ for some } t\in T \text{ such that } \psi(t)= s\}.
		$$
		Note that for all $e\in \operatorname{\mathrm{ran}}\widehat{\psi}_s$ we have
		\begin{equation*}\label{eq:def_psi_hat11}
		\widehat{\psi}^{-1}_s(e)=(et)^* \text{ where } t\in T \text{ is such that } e\leq t^+ \text{ and } \psi(t)\leq s.
		\end{equation*}
		Similarly, for $e\in \operatorname{\mathrm{ran}}\widetilde{\psi}_s$ we have 
		\begin{equation*}
		\widetilde{\psi}^{-1}_s(e)= (et)^* \text{ where } t\in T \text{ is such that } e\leq t^+ \text{ and } \psi(t)= s.
		\end{equation*}

		\begin{proposition}\label{prop:psi_hat_tilde} Let $\psi\colon T\to S$ be a proper morphism.
			\begin{enumerate}[(1)]
				\item The map $\widehat{\psi}$ is a premorphism and the triple $(\widehat{\psi}, p, P(T))$ satisfies conditions (A1)--(A4). Moreover, it satisfies condition (A3a). 
				\item The map $\widetilde{\psi}$ is a premorphism and the triple $(\widetilde{\psi}, p, P(T))$ satisfies conditions (A1)--(A4). Moreover, it satisfies condition (A3b).
			\end{enumerate}
		\end{proposition}
		
		\begin{proof} (1) We show that $\widehat{\psi}$ is a premorphism. Let $s,t\in S$ and $e\in \operatorname{\mathrm{dom}}\widehat{\psi}_s\widehat{\psi}_t$. Then $e\leq u^*$ for some $u\in \psi^{-1}(t^{\downarrow})$ and  $(ue)^+\leq v^{*}$ for some $v\in \psi^{-1}(s^{\downarrow})$. Then 
			$\psi(vu)\leq st$ and also
			$$(vu)^* = (v^*u)^* \geq ((ue)^+u)^* = (ue)^* \geq e,$$ 
			so that $e\in \operatorname{\mathrm{dom}}\widehat{\psi}_{st}$. Moreover, $\widehat{\psi}_{st}(e) = (vue)^+ = (v(ue)^+)^+ = \widehat{\psi}_s\widehat{\psi}_t(e)$. Therefore, $\widehat{\psi}_{s}\widehat{\psi}_{t}\leq \widehat{\psi}_{st}$. Thus (PM1) holds. We now show  (PM2). Let $s\in S$. By Lemma \ref{lem:well_def}(5) $\widehat{\psi}_{s^*}$ is an idempotent. Since both $\widehat{\psi}_{s^*}$ and ${\widehat{\psi}_s}^{\,*}$  are idempotents and $\operatorname{\mathrm{dom}}\widehat{\psi}_s = \operatorname{\mathrm{dom}}{\widehat{\psi}_s}^{\,*}$, it suffices to show that  $\operatorname{\mathrm{dom}}\widehat{\psi}_s \subseteq \operatorname{\mathrm{dom}}\widehat{\psi}_{s^*}$. Let $e\in \operatorname{\mathrm{dom}}\widehat{\psi}_s$. Then $e\leq t^*$ where $\psi(t)\leq s$. But then $\psi(t^*) = \psi(t)^*\leq s^*$ so that $e\in \operatorname{\mathrm{dom}}\widehat{\psi}_{s^*}$. Condition (PM3) follows from a dual argument. Condition (A1) holds by the definition of  $\widehat{\psi}$. To show~(A2) we let $e,f \in \operatorname{\mathrm{dom}}\widehat{\psi}$ be such that $e\leq f$. Then $e\leq f\leq u^*$ for some $u\in \psi^{-1}(s^{\downarrow})$. We have $\widehat{\psi}_s(e)=(ue)^+ \leq (uf)^+$. Similarly one shows that if $e\leq f\leq u^+$ for some $u\in \psi^{-1}(s^{\downarrow})$ then $(eu)^*\leq (fu)^*$. Therefore, $\widehat{\psi}_s$ is an order automorphism. Condition~(A4) is immediate since $\psi$ is surjective. We finally check that~(A3a) holds. We need to show that, for $e\in P(T)$,
			$$
			e\in \operatorname{\mathrm{dom}}\widehat{\psi}_{s} \, \text{ if and only if } \, e\in \operatorname{\mathrm{dom}}\widehat{\psi}_{u} \cap \psi^{-1}(u^*) \text{ for some }  u\leq s.
			$$
			We first assume that $e\in \operatorname{\mathrm{dom}}\widehat{\psi}_{s}$ which means that $e\leq t^*$ where $t\in \psi^{-1}(s^{\downarrow})$. Putting $u=\psi(te)$ we have $u=\psi(t)\psi(e) \leq s$, 
			$u^*=\psi (te)^* = \psi ((te)^*) = \psi (t^*e) = \psi(e)$ so that $e\in \psi^{-1}(u^*)$, and also $e= (te)^*$ where $te \in \psi^{-1}(u^{\downarrow})$ so that $e\in \operatorname{\mathrm{dom}}\widehat{\psi}_{u}$. In the reverse direction, assuming that 
			$e\in \operatorname{\mathrm{dom}}\widehat{\psi}_{u} \cap \psi^{-1}(u^*)$  for some   $u\leq s$, we have that  $e\leq r^*$ where $\psi(r)\leq u$. It follows that $e\in \operatorname{\mathrm{dom}}\widehat{\psi}_{s}$.

			Part (2) is proved similarly.
		\end{proof}
		
		Proposition \ref{prop:psi_hat_tilde}(1) and Lemma \ref{lem:prop_hat_tilde1}(1) imply that $\widehat{\psi}$ is order-preserving.
		
		The premorphisms $\widehat{\psi}$ and $\widetilde{\psi}$ will be called the {\em upper underlying premorphism} and the {\em lower underlying premorphism} of $\psi$, respectively. 
		
		\begin{lemma} \label{lem:aux15} Let $s\sim t$ and $e \in \operatorname{\mathrm{dom}}\widetilde{\psi}_s \cap \operatorname{\mathrm{dom}}\widetilde{\psi}_t$ $($resp. $e \in \operatorname{\mathrm{dom}}\widehat{\psi}_s \cap \operatorname{\mathrm{dom}}\widehat{\psi}_t)$. Then $\widetilde{\psi}_s(e) = \widetilde{\psi}_t(e)$ $($resp. $\widehat{\psi}_s(e) = \widehat{\psi}_t(e))$.
		\end{lemma}
		
		\begin{proof} Since $s\sim t$ we have $\widetilde{\psi}_s \sim \widetilde{\psi}_t \sim  \widehat{\psi}_t \sim  \widehat{\psi}_s$, by Lemma \ref{lem:prem_compatibility} and Lemma \ref{lem:well_def}(3). The statement follows since compatible partial bijections act in the same way on the intersection of their domains (see \cite[Section 1.2, Proposition 1(2)]{Lawson:book}). 
		\end{proof}

		\subsection{Decomposition of a proper extension into a partial action product} 
		
		Let $\psi\colon T\to S$ be a proper morphism between restriction semigroups. We put $p= \psi|_{P(T)}$ and let $\widehat{\psi}$ and $\widetilde{\psi}$ be the two  underlying premorphisms of $\psi$. Proposition~\ref{prop:psi_hat_tilde} ensures that we can form the partial action products $P(T)\rtimes^p_{\widehat{\psi}} S$ and $P(T)\rtimes^p_{\widetilde{\psi}} S$. Lemma \ref{lem:well_def}\eqref{ij2:3} implies that
		$P(T)\rtimes^p_{\widehat{\psi}} S = P(T)\rtimes^p_{\widetilde{\psi}} S$ as restriction semigroups.

		\begin{theorem}\label{th:isom_epsilon} Let $\psi\colon T\to S$ be a proper morphism between restriction semigroups. Then the following isomorphism holds:
			$$T \,\,\, \simeq \,\,\, P(T)\rtimes^p_{\widehat{\psi}} S \,\,\,  = \,\,\, P(T)\rtimes^p_{\widetilde{\psi}} S .$$
		\end{theorem}
		
		\begin{proof} We abbreviate $\psi\colon T\to S$ by $(\psi, T)$ and define the map ${\widehat{\eta}}_{(\psi, T)}\colon T\to P(T)\rtimes^p_{\widehat{\psi}} S$ by ${\widehat{\eta}}_{(\psi, T)}(t)=(t^+,\psi(t))$.  Applying Proposition \ref{prop:proper} to $\psi$ it follows that ${\widehat{\eta}}_{(\psi, T)}$ is injective. To show that it is surjective, let $(e,s)\in P(T)\rtimes^p_{\widehat{\psi}} S$. Then $e\in \operatorname{\mathrm{ran}}\widehat{\psi}_s$ and $\psi(e)=s^+$. Hence $e\leq t^+$ for some $t$ satisfying $\psi(t)\leq s$. But then $s^+=\psi(e)\leq \psi(t^+)\leq s^+$ yielding $\psi(e)=\psi(t^+)$. Putting $t_1=et$, we have $e=t_1^+$. In addition, $\psi(t_1)^+=s^+$ and $\psi(t_1)\leq s$ imply that $\psi(t_1)=s$. It follows that $(e,s) = (t_1^+, \psi(t_1))$, and we have shown that ${\widehat{\eta}}_{(\psi, T)}$ is a bijection. 
			 For all $t,u\in T$, applying the definition of ${\widehat{\psi}}_{\psi(t)}$, we have:
			\begin{align*}
			{\widehat{\eta}}_{(\psi, T)}(t){\widehat{\eta}}_{(\psi, T)}(u) &= (t^+, \psi(t))(u^+,\psi(u))     & \\
			&= ({\widehat{\psi}}_{\psi(t)}({\widehat{\psi}}^{-1}_{\psi(t)}(t^+)\wedge u^+), \psi(t)\psi(u))    & (\text{by } \eqref{product-in-Y-rt-S})\\
			&= ({\widehat{\psi}}_{\psi(t)}(t^*u^+), \psi(tu))  =((tt^*u^+)^+, \psi(tu))   & \\
			& = ((tu)^+, \psi(tu)) = {\widehat{\eta}}_{(\psi, T)}(tu). &  (\text{by } \eqref{eq:axioms:star} \text{ and } \eqref{eq:consequences})
			\end{align*}
			Furthermore, for all $t\in T$ we have:
			\begin{align*}
			{\widehat{\eta}}_{(\psi, T)}(t)^*  & = (t^+, \psi(t))^* = ({\widehat{\psi}}^{-1}_{\psi(t)}(t^+), \psi(t)^*)    & (\text{by } \eqref{eq:def_star_plus})\\
			&= (t^*, \psi(t^*)) = {\widehat{\eta}}_{(\psi, T)}(t^*),
			\end{align*}
			\begin{align*}
			{\widehat{\eta}}_{(\psi, T)}(t)^+  & = (t^+, \psi(t))^+ = (t^+, \psi(t)^+)    & (\text{by } \eqref{eq:def_star_plus})\\
			&= (t^+, \psi(t^+)) = {\widehat{\eta}}_{(\psi, T)}(t^+).
			\end{align*}
			It follows that ${\widehat{\eta}}_{(\psi, T)}$ is a morphism of restriction semigroups.  \end{proof}
		
		\begin{remark}  (1)  Let $S$ be an inverse semigroup, $(S, X, Y )$  a fully strict O'Carroll's triple \cite{OC77, Kh17} and $L_m(S,X,Y)$ the corresponding (inverse) $L$-semigroup. Let $\psi\colon L_m(S,X,Y)\to S$ be the induced proper morphism.  Because $\widehat{\psi}$ is order-preserving, it is globalizable \cite{GH09}, and it can be verified, along the lines of \cite[Section 4]{Kh17}, that the original O'Carroll's triple $(S, X, Y )$ can be reconstructed from $\widehat{\psi}$ applying globalization. Thus the isomorphism $T  \simeq P(T)\rtimes^p_{\widehat{\psi}} S$ given in Theorem \ref{th:isom_epsilon} provides the partial action variant of the O'Carroll's structure result \cite{OC77}.
			
			(2)   Premorphisms $S\to {\mathcal{I}}(X)$ satisfying (A1)--(A4) were considered in \cite{Kh17} for the case where $S$ is an inverse semigroup and were termed {\em fully strict} there which goes back to O'Carroll's work \cite{OC77}.
			
			(3) The underlying premorphism of a proper extension $\psi\colon T\to S$ ($T$ and $S$ being inverse semigroups) constructed in \cite[Theorem 3.4]{Kh17} coincides with our $\widetilde{\psi}$. It is noted in \cite{Kh17} that it is not always globalizable. Theorem \ref{th:isom_epsilon} shows, however, that the partial action in the formulation of \cite[Theorem 3.4]{Kh17} can always be chosen globalizable: this is the one corresponding to the premorphism $\widehat{\psi}$, which does not appear in \cite{Kh17}. Thus the claim of  Theorem \ref{th:isom_epsilon} that $T  \simeq P(T)\rtimes^p_{\widetilde{\psi}} S$ is the extension of \cite[Theorem 3.4]{Kh17} from inverse to restriction semigroups, whereas  the construction of $\widehat{\psi}$ and the claim that $T  \simeq P(T)\rtimes^p_{\widehat{\psi}} S$ are new already in the context of inverse semigroups and  provide a direct connection with O'Carroll's work \cite{OC77} (see part~(1) above).
			
		\end{remark} 
		
		\begin{remark}
			Let $S$ in the formulation of Theorem \ref{th:isom_epsilon} be a monoid. Then $T$ is a proper restriction semigroup and $S\simeq T/\sigma$. It is easy to see that  $\widehat{\psi} = \widetilde{\psi}$ coincides with the underlying premorphism of $T$ from \cite{Kud15} (which first appeared in \cite{CG} as a certain double action). Thus in this case Theorem \ref{th:isom_epsilon} specializes to the Cornock-Gould result \cite{CG} on the structure of proper restriction semigroups (see also \cite{Kud15}).
		\end{remark}
		
		\section{Classes of  proper extensions}\label{s:classes}
		Throughout this section, $\psi\colon T\to S$ is a proper morphism between restriction semigroups.
		\subsection{Order-proper extensions}
		We call $\psi$ {\em order-proper} if the premorphism $\widetilde{\psi}$ is order-preserving.
 Recall that the premorphism $\widehat{\psi}$ is always order-preserving.
				
		\begin{proposition}\label{prop:psi_coincide} Let $\psi\colon T\to S$ be a proper morphism between restriction semigroups. The following statements are equivalent:
			\begin{enumerate}[(1)]
				\item $\psi$ is order-proper;
				\item $\widetilde{\psi} = \widehat{\psi}$;
				\item  for all $s,t\in S$ such that $s\leq t$ and all $u\in \psi^{-1}(s)$ there is $v\in \psi^{-1}(t)$ such that $u\leq v$;
				\item $\psi^{-1}(s^{\downarrow}) = (\psi^{-1}(s))^{\downarrow}$, for all $s\in S$.
			\end{enumerate}
		\end{proposition}
		
		\begin{proof} (1) $\Rightarrow$ (2) We assume that $\widetilde{\psi}$ is order-preserving and let $s\in S$.  It is enough to show that $\operatorname{\mathrm{dom}}\widehat{\psi}_s \subseteq \operatorname{\mathrm{dom}}\widetilde{\psi}_s$. Let $e\in \operatorname{\mathrm{dom}}\widehat{\psi}_s$. Then $e\leq u^*$ for some $u\in T$ satisfying $\psi(u)\leq s$. Then $e\in  \operatorname{\mathrm{dom}}\widetilde{\psi}_{\psi(u)}$. The assumption implies that $\operatorname{\mathrm{dom}}\widetilde{\psi}_{\psi(u)}\subseteq \operatorname{\mathrm{dom}}\widetilde{\psi}_{s}$ whence $e\in \operatorname{\mathrm{dom}}\widetilde{\psi}_{s}$, as needed.
			
			(2) $\Rightarrow$ (3) 
			Let $s,t\in S$ be such that $s\leq t$ and let $u\in \psi^{-1}(s)$. Then $u^*\in \operatorname{\mathrm{dom}}\widetilde{\psi}_s = \operatorname{\mathrm{dom}}\widehat{\psi}_s$. It follows that $u^*\in \operatorname{\mathrm{dom}}\widehat{\psi}_t = \operatorname{\mathrm{dom}}\widetilde{\psi}_t$, which shows that $u^* \leq v^*$ for some $v\in \psi^{-1}(t)$. Since $\psi(v) \sim \psi(u)$, we also have $v\sim u$, by Lemma \ref{lem:proper_compatibility}. Therefore $u\leq v$.
			
			(3) $\Rightarrow$ (4) It is easy to see that the inclusion $(\psi^{-1}(s))^{\downarrow}\subseteq \psi^{-1}(s^{\downarrow})$ always holds and assuming part (3), the reverse inclusion holds, too.
			
			(4) $\Rightarrow$ (1) We assume that condition in part (4) holds and let us show that $\widetilde{\psi}$ is order-preserving. Let $s,t\in S$ be such that $s\leq t$. We show that $\widetilde{\psi}_s\leq \widetilde{\psi}_t$. In view of Lemma \ref{lem:aux15}, it  suffices to show that $\operatorname{\mathrm{dom}}\widetilde{\psi}_s\subseteq  \operatorname{\mathrm{dom}}\widetilde{\psi}_t$. Let $e\in \operatorname{\mathrm{dom}}\widetilde{\psi}_s$. Then $e\leq u^*$ for some $u\in \psi^{-1}(s)$. By assumption, there is $v\in \psi^{-1}(t)$ satisfying $u\leq v$. Since $e\leq u^*\leq v^*$ it follows that $e\in \operatorname{\mathrm{dom}}\widetilde{\psi}_t$, as needed. 
		\end{proof}

		\subsection{Extra proper extensions}
		
		We establish a connection between local strongness of  $\widehat{\psi}$ and $\widetilde{\psi}$.
		
		\begin{theorem}\label{th:extra_locally_strong}
			Let $\psi\colon T\to S$ be a proper morphism between restriction semigroups. The following statements are equivalent:
			\begin{enumerate}[(1)]
				\item $\widehat{\psi}$  satisfies condition (LSr) (respectively (LSl));
				\item $\widehat{\psi}$  satisfies condition (Sr) (respectively (Sl));
				\item $\widetilde{\psi}$ satisfies condition (LSr) (respectively (LSl)).
				
			\end{enumerate}
		\end{theorem}
		
		\begin{proof} Since $\widehat{\psi}$ is order-preserving, the equivalence (1) $\Longleftrightarrow$ (2) follows from Theorem \ref{th:strong_restr}.
			
			(1) $\Rightarrow$ (3)
			We assume that $\widehat{\psi}$ satisfies (LSr) and show that so does $\widetilde{\psi}$. For $s,t\in S$ we show that $\widetilde{\psi}_{st^+}\widetilde{\psi}_t = \widetilde{\psi}_{st}\widetilde{\psi}_t^{\,*}$. Note that $\widetilde{\psi}_{st^+}\widetilde{\psi}_t = \widetilde{\psi}_{st^+}\widetilde{\psi}_t\widetilde{\psi}_t^{\,*} \leq \widetilde{\psi}_{st}\widetilde{\psi}_t^{\,*}$ holds because $\widetilde{\psi}$ is a premorphism. It is thus enough to prove that 
			$$\operatorname{\mathrm{dom}}\widetilde{\psi}_{st}\widetilde{\psi}_t^{\,*} \subseteq \operatorname{\mathrm{dom}}\widetilde{\psi}_{st^+}\widetilde{\psi}_t.$$
			Let $e\in \operatorname{\mathrm{dom}}\widetilde{\psi}_{st}\widetilde{\psi}_t^{\,*}$. Since $\operatorname{\mathrm{dom}}\widetilde{\psi}_t^{\,*} = \operatorname{\mathrm{dom}}\widetilde{\psi}_t$ and $\widetilde{\psi}_t^{\,*}$ acts identically on its domain, it follows that $e\in \operatorname{\mathrm{dom}}\widetilde{\psi}_t \cap \operatorname{\mathrm{dom}}\widetilde{\psi}_{st}$. Then $e\leq u^*$ where $\psi(u) = t$ and $e\leq v^*$ where $\psi(v)=st$. Hence $e\leq u^*v^* = (uv^*)^*=(vu^*)^*$. Put $u_1=uv^*$ and $v_1 = vu^*$. Then $e\leq u_1^*= v_1^*$ and $\psi(u_1)= t(st)^*$, $\psi(v_1)=st$. If we show that $u_1^* \in \operatorname{\mathrm{dom}}\widetilde{\psi}_{st^+}\widetilde{\psi}_t$ then $e\in \operatorname{\mathrm{dom}}\widetilde{\psi}_{st^+}\widetilde{\psi}_t$ as well because $\operatorname{\mathrm{dom}}\widetilde{\psi}_{st^+}\widetilde{\psi}_t$ is an order ideal. Therefore, without loss of generality we can assume that $e=u_1^*= v_1^*$. To show that $e\in \operatorname{\mathrm{dom}}\widetilde{\psi}_{st^+}\widetilde{\psi}_t$ it suffices to show that $\widetilde{\psi}_t(e) \in \operatorname{\mathrm{dom}}\widetilde{\psi}_{st^+}$. Since $\widehat{\psi}$ satisfies condition (LSr) and $e\in \operatorname{\mathrm{dom}}\widehat{\psi}_{st}\widehat{\psi}_t^{\,*}$ it follows that
			$e\in \operatorname{\mathrm{dom}}\widehat{\psi}_{st^+}\widehat{\psi}_t$, that is, $\widetilde{\psi}_t(e) = \widehat{\psi}_t(e) \in \operatorname{\mathrm{dom}}\widehat{\psi}_{st^+}$. By definition, this means that $\widetilde{\psi}_t(e)\leq w^*$ where $\psi(w)\leq st^+$.
			Let $p=\psi(w)$. Note that $pt\leq st$. We have $\widetilde{\psi}_{st}(e) = \widehat{\psi}_{st}(e) = \widehat{\psi}_{st^+}\widetilde{\psi}_t(e)$, that is, $v_1^+ = (wu_1)^+$. Because $\psi(wu_1) \leq st = \psi(v_1)$, we have $\psi(v_1) \sim \psi(wu_1)$. By Lemma \ref{lem:proper_compatibility} this yields $v_1\sim wu_1$ whence $v_1=wu_1$. But then $st=\psi(v_1) = \psi(wu_1)= pt (st)^*=pt$.  It follows that  $(st^+)^+ = (pt^+)^+$ whence $st^+=pt^+$ thus $p \leq st^+ = pt^+ \leq p$. Hence $p=st^+$ and $\widetilde{\psi}_t(e)\in \operatorname{\mathrm{dom}}\widetilde{\psi}_{st^+}$, as needed. 
			
			(3) $\Rightarrow$ (1) We assume that $\widetilde{\psi}$ satisfies (LSr) and show that so does $\widehat{\psi}$. Let $s,t\in S$. Just as before, it is thus enough to prove that 
			$$\operatorname{\mathrm{dom}}\widehat{\psi}_{st}\widehat{\psi}_t^{\,*} \subseteq \operatorname{\mathrm{dom}}\widehat{\psi}_{st^+}\widehat{\psi}_t.$$
			Let $e\in \operatorname{\mathrm{dom}}\widehat{\psi}_{st}\widehat{\psi}_t^{\,*}$. 
			Then $e\in \operatorname{\mathrm{dom}}\widehat{\psi}_t \cap \operatorname{\mathrm{dom}}\widehat{\psi}_{st}$, so that $e\leq u^*$ where $\psi(u) = t_1\leq  t$ and $e\leq v^*$ where $\psi(v) \leq st$.  Then $\psi(v)=stf = ss^*tf$ for some $f\in P(S)$. We put $s^*tf = t_2$ and note that $s^*t_2 = t_2$. Because $t_1\sim t_2$, we have
			$t_1t_2^* = t_2t_1^*$, denote this element by $t_3$. Let $u_1=uv^*$ and $v_1=vu^*$. Then $e\leq u_1^*$ where $\psi(u_1)=\psi(uv^*) = t_1(st_2)^* = t_1(st_2)^*t_2^*= t_3(st_2)^* =  t_3(s^*t_2)^* = t_3t_2^* = t_3$ and $e\leq v_1^*$ with $\psi(v_1) = st_2t_1^* = st_3$. Therefore,
			$e\in  \operatorname{\mathrm{dom}}\widetilde{\psi}_{t_3} \cap \operatorname{\mathrm{dom}}\widetilde{\psi}_{st_3}$ so that $e\in \operatorname{\mathrm{dom}}\widetilde{\psi}_{st_3}\widetilde{\psi}_{t_3}^*$. From  $\widetilde{\psi}_{st_3} \widetilde{\psi}_{t_3}^* = \widetilde{\psi}_{s t_3^+}\widetilde{\psi}_{t_3}$ it now follows that
			$$
			\widehat{\psi}_{t}(e) = \widetilde{\psi}_{t_3}(e) \in  \operatorname{\mathrm{dom}}\widetilde{\psi}_{st_3^+} \subseteq \operatorname{\mathrm{dom}}\widehat{\psi}_{st_3^+}.
			$$ 
			Since $\widehat{\psi}$ is order-preserving and $st_3^+ \leq st^+$, we obtain $\widehat{\psi}_{t}(e) \in  \operatorname{\mathrm{dom}}\widehat{\psi}_{st^+}$, as needed.
		\end{proof}
		
		We say that $\psi$ is {\em extra proper} provided that $\widetilde{\psi}$ is locally strong (or, equivalently, $\widehat{\psi}$ is strong). Extra proper morphisms form a class narrower than that of proper morphisms which still generalizes the class  of idempotent pure morphisms between inverse semigroups (since $\widetilde{\psi}$ is always inverse when $S$ and $T$ are inverse semigroups). In the case where $S$ is a monoid and $\psi$ is extra proper,  $T$ is an extra proper restriction  semigroup \cite{CG, Kud19}. It follows that extra proper morphisms generalize idempotent pure morphisms between inverse semigroups, on the one hand, and extra proper restriction semigroups, on the other hand.

		\subsection{Perfect extensions}\label{subs:perfect} We now establish a connection between local multiplicativity of $\widehat{\psi}$ and $\widetilde{\psi}$.
		\begin{theorem} \label{th:perfect} Let $\psi\colon T\to S$ be a proper moprhism. 
			The following statements are equivalent:
			\begin{enumerate}[(1)]
				\item $\widehat{\psi}$ is multiplicative;
				\item $\widehat{\psi}$ is locally multiplicative;
				\item $\widetilde{\psi}$ is locally multiplicative;
				\item  the equality $\psi^{-1}((st)^{\downarrow}) = \psi^{-1}(s^{\downarrow})\psi^{-1}(t^{\downarrow})$ holds for all $s,t\in S$.
			\end{enumerate}
		\end{theorem}
		
		\begin{proof} 
			The equivalence (1) $\Leftrightarrow$ (2) follows from Proposition \ref{prop:local2}.
			
			(1) $\Rightarrow$ (4) Since the inclusion $\psi^{-1}(s^{\downarrow})\psi^{-1}(t^{\downarrow}) \subseteq \psi^{-1}((st)^{\downarrow})$ clearly always holds, we prove the reverse inclusion, under the assumption that  $\widehat{\psi}$ is  multiplicative. Let $p\in \psi^{-1}((st)^{\downarrow})$. Then  $\psi(p)\leq st$.
			Observe that $p^*\in \operatorname{\mathrm{dom}}\widehat{\psi}_{st}$.
			Then by assumption $p^*\in \operatorname{\mathrm{dom}}\widehat{\psi}_{s}\widehat{\psi}_{t}$, that is, $p^*\in \operatorname{\mathrm{dom}}\widehat{\psi}_{t}$ and $\widehat{\psi}_{t}(p^*) \in \operatorname{\mathrm{dom}}\widehat{\psi}_{s}$. This means that there is $q\in \psi^{-1}(t^{\downarrow})$ such that $p^*\leq q^*$. Furthermore, there is $r\in \psi^{-1}(s^{\downarrow})$ satisfying $\widehat{\psi}_{t}(p^*) = (qp^*)^+\leq r^{*}$. Observe that $\psi(rq),\psi(p)\leq st$ whence, in view of Lemma \ref{lem:proper_compatibility}, $rq \sim p$. 
			Because, in addition,
			\begin{align*}
			(rq)^* = & (r^*q)^* \geq ((qp^*)^+q)^* & \\
			& = (qp^*)^* & (\text{since } (qp^*)^+q = qp^* \text{ by } \eqref{eq:mov_proj})\\
			& = q^*p^* = p^*,
			\end{align*}
			we obtain $rq \geq p$. Then $p=rqp^*$. Since $r\in \psi^{-1}(s^{\downarrow})$ and $qp^*\in \psi^{-1}(t^{\downarrow})$ it follows that $p\in \psi^{-1}(s^{\downarrow})\psi^{-1}(t^{\downarrow})$, as desired.
			
			(4) $\Rightarrow$ (3) We assume that (4) holds and let $s,t\in S$. Since $\widetilde{\psi}_{st^+}\widetilde{\psi}_{s^*t} \leq \widetilde{\psi}_{st}$, we prove only the inclusion
			$\operatorname{\mathrm{dom}}\widetilde{\psi}_{st}\subseteq \operatorname{\mathrm{dom}}\widetilde{\psi}_{st^+}\widetilde{\psi}_{s^*t}$. Let $e\in \operatorname{\mathrm{dom}}\widetilde{\psi}_{st}$. This means that $e\leq p^*$ for some $p\in \psi^{-1}(st) = \psi^{-1}((st^+)(s^*t))$. The assumption yields that $p\in \psi^{-1}((st^+)^{\downarrow})\psi^{-1}((s^*t)^{\downarrow})$, that is, $p= rq$ where $\psi(r)\leq st^+$ and $\psi(q) \leq s^*t$. Let $r_1=rq^+$ and $q_1=r^*q$. Then
			$p=r_1q_1$ and the equalities $p^*=q_1^*$, $p^+= r_1^+$, $q_1^+ = r_1^*$ hold.  In addition, $$st = \psi(r_1q_1) = \psi(r_1)\psi(q_1) \leq \psi(r)\psi(q) \leq st^+s^*t = st,$$
			so that $\psi(r_1)\psi(q_1) = st$. But then 
			$$(s^*t)^*=(st)^* = (\psi(r_1)\psi(q_1))^* = (\psi(r_1)^*\psi(q_1))^* = (\psi(q_1)^+\psi(q_1))^* = \psi(q_1)^*,$$
			implying that $\psi(q_1)=s^*t$. Similarly, $\psi(r_1) = st^+$. Because $e\leq p^* = (q_1)^*$ and $\psi(q_1) = s^*t$, it follows that that $e \in \operatorname{\mathrm{dom}}\widetilde{\psi}_{s^*t}$ and, furthermore,
			$\widetilde{\psi}_{s^*t}(e) =  (q_1e)^+ \leq q_1^+ = r_1^*$. Hence $\widetilde{\psi}_{s^*t}(e) \in \operatorname{\mathrm{dom}}\widetilde{\psi}_{st^+}$.
			
			(3) $\Rightarrow$ (1) We assume (3) holds
			and let $s,t\in S$. Since $\widehat{\psi}_s\widehat{\psi}_t \leq \widehat{\psi}_{st}$, we prove only the inclusion
			$\operatorname{\mathrm{dom}}\widehat{\psi}_{st}\subseteq \operatorname{\mathrm{dom}}\widehat{\psi}_s\widehat{\psi}_t$.  Consider any element $e\in \operatorname{\mathrm{dom}}\widehat{\psi}_{st}$ and note that $e \leq p^*$ where  $u=\psi(p)\leq st$. Then $e\in \operatorname{\mathrm{dom}}\widetilde{\psi}_u$. Let $r=s(tu^*)^+$ and $q=s^*(tu^*)$. We have $u=rq$ where $r\leq s$, $q \leq t$ and $r^*=q^+$. It follows from (3) that $e\in \operatorname{\mathrm{dom}}\widetilde\psi_r\widetilde\psi_{q}$. Since $\widetilde\psi_{r}\leq \widehat\psi_{r}$, $\widetilde\psi_{q}\leq \widehat\psi_{q}$ and $\widehat\psi$ is order-preserving, we have $e\in \operatorname{\mathrm{dom}}\widehat\psi_{r}\widehat\psi_{q}\subseteq\operatorname{\mathrm{dom}}\widehat\psi_{s}\widehat\psi_{t}$, as needed.
		\end{proof}
		
		We say that $\psi$ is {\em perfect} provided that $\widehat{\psi}$ is multiplicative (and thus all the equivalent conditions of Theorem \ref{th:perfect} hold). 
		
		If $S$ is a monoid then $\psi\colon T\to S$ is perfect if and only if $T$ is an almost perfect restriction semigroup \cite{Jones16} (termed an ultra proper restriction semigroup in \cite{Kud15}). It follows that perfect extensions generalize almost perfect restriction semigroups.
		
		\section{Categories of proper extensions of $S$ and of partial actions of $S$}\label{s:categorical}
		
		Throughout this section, we fix $S$ to be a restriction semigroup. 
		
		\subsection{The categories of premorphisms from $S$}\label{subs:domains}
		We define the category ${\mathcal{A}}(S)$ as follows. Its objects are triples $(\alpha, p, X)$ where $X$ is a semilattice,  $p\colon X\to P(S)$ a morphism of semilattices and $\alpha\colon S\to {\mathcal{I}}(X)$ a premorphism, such that conditions (A1)--(A4) are satisfied. A {\em morphism} from  $(\alpha, p_{\alpha}, X)$ to $(\beta, p_{\beta}, Y)$ is a semilattice morphism
		$f\colon X\to Y$ such that the following conditions hold:
		\begin{enumerate}
			\item[(M1)] $p_{\alpha}=p_{\beta}f$;
			\item[(M2)] $f(\operatorname{\mathrm{dom}}\alpha_s) \subseteq \operatorname{\mathrm{dom}}\beta_s$ and $\beta_s(f(e)) = f(\alpha_s(e))$ for all $s\in S$ and $e\in \operatorname{\mathrm{dom}}\alpha_s$.
		\end{enumerate} 
		
		Condition (M2) points out that our notion of a morphism between partial actions agrees with  that in the sense of Abadie~\cite{Abadie03}. It is easy to verify that ${\mathcal{A}}(S)$  is indeed a category.
		
		In Subsection \ref{subs:equiv} we will need the following lemma.
		
		\begin{lemma}\label{lem:tech} A morphism $f\colon X\to Y$ in the category ${\mathcal{A}}(S)$ satisfies condition:
			\begin{enumerate}
				\item[{\em (M2r)}] $f(\operatorname{\mathrm{ran}}\alpha_s)\subseteq \operatorname{\mathrm{ran}}\beta_s$ and $\beta_s^{-1}(f(e)) = f(\alpha_s^{-1}(e))$ for all $e\in \operatorname{\mathrm{ran}}\alpha_s$.
			\end{enumerate}
		\end{lemma}
		
		\begin{proof} Let $s\in S$ and $e\in \operatorname{\mathrm{ran}}\alpha_s$. Then  $\alpha_s^{-1}(e) \in \operatorname{\mathrm{dom}}\alpha_s$. From (M2) we have $f(\alpha_s^{-1}(e)) \in \operatorname{\mathrm{dom}}\beta_s$ and $\beta_s(f(\alpha_s^{-1}(e))) = f(\alpha_s\alpha_s^{-1}(e)) = f(e)$. Hence $f(e)\in \operatorname{\mathrm{ran}}\beta_s$ and $f(\alpha_s^{-1}(e)) = \beta_s^{-1}(f(e))$, as  required.
		\end{proof}
		
		It is easy to see that (M2) is in fact equivalent to (M2r).

		We now define two subcategories of the category ${\mathcal{A}}(S)$. Their objects arise from premorphisms underlying proper extensions of $S$ (see Proposition \ref{prop:psi_hat_tilde}). Namely, we put ${\widehat{\mathcal{A}}}(S)$ to be the full subcategory of  ${\mathcal{A}}(S)$ whose objects satisfy condition (A3a), and ${\widetilde{\mathcal{A}}}(S)$ to be the full subcategory of  ${\mathcal{A}}(S)$ whose objects satisfy condition (A3b).

		Let  ${\widehat{I}}\colon {\widehat{\mathcal{A}}}(S) \to {\mathcal{A}}(S)$ and ${\widetilde{I}}\colon {\widetilde{\mathcal{A}}}(S) \to {\mathcal{A}}(S)$ be the inclusion functors.  We now construct functors in the reverse directions.
		Let $(\varphi, p,X)$ be an object of ${\mathcal{A}}(S)$.   
		
		For each $s\in S$  
		we define ${\widehat{\varphi}}_s \in {\mathcal{I}}(X)$ by $$\operatorname{\mathrm{dom}}{\widehat{\varphi}}_s = \{x\colon x\in \operatorname{\mathrm{dom}}\varphi_t \text{ for some } t\leq s\}$$
		and ${\widehat{\varphi}}_s(x) = \varphi_t(x)$, where $t\leq s$ is such that $x\in \operatorname{\mathrm{dom}}\varphi_t$. By  Lemma  \ref{lem:prem_compatibility} this is well defined. It is easy to check that this defines an object $({\widehat{\varphi}},p,X)$ of ${\widehat{\mathcal{A}}}(S)$. We put ${\widehat{F}}(\varphi, p,X) = ({\widehat{\varphi}},p,X)$. This  assignment gives rise to a functor ${\widehat{F}}\colon {\mathcal{A}}(S)\to {\widehat{\mathcal{A}}}(S)$ which we call the {\em extension of domains} functor.
		
		For each $s\in S$  
		we now define ${\widetilde{\varphi}}_s \in {\mathcal{I}}(X)$ by restricting $\varphi_s$ to the set
		$$\operatorname{\mathrm{dom}}{\widetilde{\varphi}}_s=\{x\in \operatorname{\mathrm{dom}}\varphi_s\colon p(x) = s^*\}^{\downarrow}.$$
		This defines an object $({\widetilde{\varphi}},p,X)$ of ${\widetilde{\mathcal{A}}}(S)$. We put ${\widetilde{F}}(\varphi, p,X) = ({\widetilde{\varphi}},p,X)$. This assignment gives rise to a functor ${\widetilde{F}}\colon {\mathcal{A}}(S)\to {\widetilde{\mathcal{A}}}(S)$ which we call the  {\em restriction of domains} functor.
		
		\begin{theorem}\label{th:adjunctions} Let $S$ be a restriction semigroup. 
			\begin{enumerate}[(1)]
				\item The functor ${\widehat{F}}\colon {\mathcal{A}}(S) \to {\widehat{\mathcal{A}}}(S)$ is a left adjoint to the functor ${\widehat{I}}\colon {\widehat{\mathcal{A}}}(S) \to {\mathcal{A}}(S)$.
				\item The functor ${\widetilde{F}}\colon {\mathcal{A}}(S) \to {\widetilde{\mathcal{A}}}(S)$ is a right adjoint to the functor ${\widetilde{I}}\colon {\widetilde{\mathcal{A}}}(S) \to {\mathcal{A}}(S)$.
				\item The functors  ${\widetilde{F}}{\widehat{I}} \colon {\widehat{\mathcal{A}}}(S)\to {\widetilde{\mathcal{A}}}(S)$ and ${\widehat{F}}{\widetilde{I}}\colon {\widetilde{\mathcal{A}}}(S)\to {\widehat{\mathcal{A}}}(S)$ establish an isomorphism between the categories ${\widehat{\mathcal{A}}}(S)$ and ${\widetilde{\mathcal{A}}}(S)$.
			\end{enumerate}
		\end{theorem}
		
		\begin{proof} (1)  Let $(\varphi, p_{\varphi}, X)$ be an object of ${\mathcal{A}}(S)$. Then $\eta_{(\varphi, p_{\varphi}, X)}={\mathrm{id}}_X$ is a morphism from  $(\varphi, p_{\varphi}, X)$ to $(\widehat{\varphi}, p_{\varphi}, X)=\widehat{I}\widehat{F}(\varphi, p_{\varphi}, X)$  in ${\mathcal{A}}(S)$ which is natural in $(\varphi, p_{\varphi}, X)$. Moreover, for any
			object $(\psi, p_{\psi}, Y)$ in  ${\widehat{\mathcal{A}}}(S)$ and any morphism $f\colon X \to Y$ from $(\varphi, p_{\varphi}, X)$ to $(\psi, p_{\psi}, Y)=\widehat{I}(\psi, p_{\psi}, Y)$ in ${\mathcal{A}}(S)$ the map $f$ is the unique morphism from $(\widehat{\varphi}, p_{\varphi}, X)=\widehat{F}(\varphi, p_{\varphi}, X)$ to $(\psi, p_{\psi}, Y)$ such that the diagram below commutes:
			\begin{center}
				\begin{tikzpicture}
				\node (psi) at (0,0) {$(\psi, p_{\psi}, Y)$};
				\node[left = 0.8 cm of psi] (varphi) {$(\varphi, p_{\varphi}, X)$};
				\node[below = 1 cm of varphi] (varphi_tilde) {$(\widehat{\varphi}, p_{\varphi}, X)$};
				\draw[<-] (psi)--(varphi) node [midway,above] {$f$};
				\draw[<-,dashed] (psi)--(varphi_tilde) node [midway,below] {$f$};
				\draw[<-] (varphi_tilde)--(varphi) node [midway,left] {$\eta_{(\varphi, p_{\varphi}, X)}$};
				\end{tikzpicture}
			\end{center} 
			
			(2) Let $(\varphi, p_{\varphi}, X)$ be an object of ${\mathcal{A}}(S)$. Then $\varepsilon_{(\varphi, p_{\varphi}, X)}={\mathrm{id}}_X$ is a morphism from $(\widetilde{\varphi}, p_{\varphi}, X)=\widetilde{I}\widetilde{F}(\varphi, p_{\varphi}, X)$ to $(\varphi, p_{\varphi}, X)$  in ${\mathcal{A}}(S)$ which is natural in $(\varphi, p_{\varphi}, X)$. Moreover, for any
			$(\psi, p_{\psi}, Y)$ in ${\widetilde{\mathcal{A}}}(S)$ and any morphism $f\colon Y \to X$ from $(\psi, p_{\psi}, Y)=\widetilde{I}(\psi, p_{\psi}, Y)$  to $(\varphi, p_{\varphi}, X)$ in ${\mathcal{A}}(S)$ the map $f$ is the unique morphism from $(\psi, p_{\psi}, Y)$ to $(\widetilde{\varphi}, p_{\varphi}, X)=\widetilde{F}(\varphi, p_{\varphi}, X)$ such that the diagram below commutes:
			
			\begin{center}
				\begin{tikzpicture}
				\node (psi) at (-1,0) {$(\psi, p_{\psi}, Y)$};
				\node[right= 0.8 cm of psi] (varphi) {$(\varphi, p_{\varphi}, X)$};
				\node[above= 1 cm of varphi] (varphi_hat) {$(\widetilde{\varphi}, p_{\varphi}, X)$};
				\draw[->] (psi)--(varphi) node [midway,below] {$f$};
				\draw[->,dashed] (psi)--(varphi_hat) node [midway,above] {$f$};
				\draw[->] (varphi_hat)--(varphi) node [midway,right] {$\varepsilon_{(\varphi, p_{\varphi}, X)}$};
				\end{tikzpicture}
			\end{center} 
			
			Part (3) follows from the observation that for any $(\varphi, p_{\varphi}, X)$ in ${\widehat{\mathcal{A}}}(S)$ and any $(\psi, p_{\psi}, Y)$ in ${\widetilde{\mathcal{A}}}(S)$ we have $(\varphi, p_{\varphi}, X)={\widehat{F}}{\widetilde{I}}{\widetilde{F}}{\widehat{I}}(\varphi, p_{\varphi}, X)$ and $(\psi, p_{\psi}, Y) = {\widetilde{F}}{\widehat{I}}{\widehat{F}}{\widetilde{I}}(\psi, p_{\psi}, Y)$. 
		\end{proof}
		
		\begin{corollary}\label{cor:cor2} Let $S$ be a restriction semigroup. Then:
			\begin{enumerate}[(1)]
				\item ${\widehat{\mathcal{A}}}(S)$ is a reflective subcategory of ${\mathcal{A}}(S)$, the reflector being the functor ${\widehat{F}}$;
				\item  ${\widetilde{\mathcal{A}}}(S)$ is a coreflective subcategory of ${\mathcal{A}}(S)$, the coreflector being the functor ${\widetilde{F}}$.
			\end{enumerate}
		\end{corollary}
		
		Let ${\mathcal{A}}_s(S)$, ${\widetilde{\mathcal{A}}}_s(S)$ and ${\widehat{\mathcal{A}}}_s(S)$ be the subcategories of  ${\mathcal{A}}(S)$, ${\widetilde{\mathcal{A}}}(S)$ and ${\widehat{\mathcal{A}}}(S)$, respectively, whose morphisms satisfy the following additional requirement:
		\begin{enumerate}
			\item[(M3)] ${\operatorname{\mathrm{dom}}\beta_s}\cap p_{\beta}^{-1}(s^*) = f(\operatorname{\mathrm{dom}}\alpha_s \cap p_{\alpha}^{-1}(s^*))$ for all $s\in S$.
		\end{enumerate}
		One can check that  (M3) is equivalent to the condition
		\begin{enumerate}
			\item[{(M3r)}] $\operatorname{\mathrm{ran}}\beta_s\cap p_{\beta}^{-1}(s^+) =  f(\operatorname{\mathrm{ran}}\alpha_s \cap p_{\alpha}^{-1}(s^+))$ for all $s\in S$.
		\end{enumerate}
		
		\begin{remark} \label{rem:n19} We note that for the morphisms of the categories ${\widetilde{\mathcal{A}}}_s(S)$ and ${\widehat{\mathcal{A}}}_s(S)$, each of (M3) and (M3r) is equivalent to $f(\operatorname{\mathrm{dom}}\alpha_s) = \operatorname{\mathrm{dom}}\beta_s$ for all $s\in S$.
		\end{remark}

		Let ${\widehat{F}}_s$, ${\widetilde{F}}_s$ be the restrictions of the functors ${\widehat{F}}$ and ${\widetilde{F}}$, respectively, to the category ${\mathcal{A}}_s(S)$; and ${\widetilde{I}}_s$, ${\widehat{I}}_s$ be the restrictions of the functors ${\widetilde{I}}$, ${\widehat{I}}$ to the categories  ${\widetilde{\mathcal{A}}}_s(S)$ and  ${\widehat{\mathcal{A}}}_s(S)$, respectively. The following is an immediate consequence of Theorem \ref{th:adjunctions}.
		
		\begin{corollary}\label{cor:adjunctions} Let $S$ be a restriction semigroup.
			\begin{enumerate}[(1)]
				\item The functor ${\widehat{F}}_s\colon {\mathcal{A}}_s(S) \to {\widehat{\mathcal{A}}}_s(S)$ is a left adjoint to the functor ${\widehat{I}}_s\colon {\widehat{\mathcal{A}}}_s(S) \to {\mathcal{A}}_s(S)$.
				\item The functor ${\widetilde{F}}_s\colon {\mathcal{A}}_s(S) \to {\widetilde{\mathcal{A}}}_s(S)$ is a right adjoint to the functor ${\widetilde{I}}_s\colon {\widetilde{\mathcal{A}}}_s(S) \to {\mathcal{A}}_s(S)$.
				\item The functors  ${\widetilde{F}}_s{\widehat{I}}_s \colon {\widehat{\mathcal{A}}}_s(S)\to {\widetilde{\mathcal{A}}}_s(S)$ and ${\widehat{F}}_s{\widetilde{I}}_s\colon {\widetilde{\mathcal{A}}}_s(S)\to {\widehat{\mathcal{A}}}_s(S)$ establish an isomorphism between the categories ${\widehat{\mathcal{A}}}_s(S)$ and ${\widetilde{\mathcal{A}}}_s(S)$.
			\end{enumerate}
		\end{corollary}
		
		A statement analogous to Corollary \ref{cor:cor2} also follows.

		\begin{remark}\label{rem:trivial} Let $S$ be a monoid and $(\varphi, p, X)$ an object of ${\mathcal{A}}(S)$. Then $p$ is the only morphism $X\to P(S)$. Note that condition (A3) says that $\varphi$ is unital, and, moreover,  (A3a) and (A3b) trivially hold. It follows that ${\mathcal{A}}(S) = {\widetilde{\mathcal{A}}}(S) = {\widehat{\mathcal{A}}}(S)$,  and, similarly, ${\mathcal{A}}_s(S) = {\widetilde{\mathcal{A}}}_s(S) = {\widehat{\mathcal{A}}}_s(S)$, and in this case the statements of Theorem \ref{th:adjunctions}, Corollary \ref{cor:adjunctions} and Corollary \ref{cor:cor2} become trivial.
		\end{remark}

		\subsection{Equivalence of categories of proper extensions and partial actions of $S$} \label{subs:equiv}
		Let ${\mathcal{P}}(S)$ be the category whose objects are proper morphisms $\psi\colon T\to S$ where $T$ is a restriction semigroup (and $S$ is our fixed restriction semigroup). A {\em morphism} from $\psi_1\colon T_1\to S$ to $\psi_2\colon T_2\to S$ is a morphism $\gamma\colon T_1\to T_2$ of restriction semigroups such that $\psi_2\gamma = \psi_1$.
		
		We construct a functor $U\colon  {\mathcal{A}}(S)\to {\mathcal{P}}(S)$. Let $(\alpha, p_{\alpha}, X)$ be an object of ${\mathcal{A}}(S)$.  We put $U(\alpha, p_{\alpha}, X)$ to be the proper morphism $X \rtimes^{p_{\alpha}}_{\alpha} S \to S$, $(x,s)\mapsto s$. Let now $p\colon X\to Y$ be a morphism from $(\alpha, p_{\alpha}, X)$  to $(\beta, p_{\beta}, Y)$. We define $U(p) \colon X \rtimes^{p_{\alpha}}_{\alpha} S \to Y \rtimes^{p_{\beta}}_{\beta} S$ by the assignment $(x,s)\mapsto (p(x),s)$. By  (M1) and (M2r) this is well defined. We show that $U(p)$ is a morphism of restriction semigroups. Let $(x,s), (y,t)$ $\in X \rtimes^{p_{\alpha}}_{\alpha} S$. We have $U(p)((x,s)(y,t)) = (p(\alpha_s(\alpha_s^{-1}(x)\wedge y)), st)$ and 
		$$U(p)(x,s)U(p)(y,t) = (p(x),s)(p(y),t) = (\beta_s(\beta_s^{-1}(p(x))\wedge p(y)), st).$$  
		It suffices to show that 
		$p(\alpha_s(\alpha_s^{-1}(x)\wedge y)) = \beta_s(\beta_s^{-1}(p(x))\wedge p(y))$. The left-hand side of the latest equality is equal to
		$\beta_s(p(\alpha_s^{-1}(x)\wedge y))$. It is thus enough to verify the equality $p(\alpha_s^{-1}(x)\wedge y)=\beta_s^{-1}(p(x))\wedge p(y)$. The latter equality easily follows because $p$ is semilattice  morphism and since $p(\alpha_s^{-1}(x)) =  \beta_s^{-1}(p(x))$ by (M2r). That $U(p)$ respects the unary operations $^*$ and $^+$ is easier to show, and we leave this to the reader. It follows that $U$ is a functor. We denote $\widehat{U} = U\widehat{I} \colon {\widehat{\mathcal{A}}}(S)\to {\mathcal{P}}(S)$ and $\widetilde{U} = U\widetilde{I} \colon {\widetilde{\mathcal{A}}}(S)\to {\mathcal{P}}(S)$.

		The following is an immediate consequence of (M3r).
		\begin{lemma}\label{lem:n1} If a morphism $p\colon X\to Y$ from $(\alpha, p_{\alpha}, X)$ to $(\beta, p_{\beta}, Y)$ satisfies (M3)  then  the map $U(p)$ is surjective.
		\end{lemma}

		We now construct a functor ${\widehat{G}}\colon {\mathcal{P}}(S) \to {\widehat{\mathcal{A}}}(S)$. Let $\psi\colon T\to S$ be an object of  ${\mathcal{P}}(S)$. We define ${\widehat{G}}(\psi)=(\widehat{\psi}, \psi|_{P(T)}, P(T))$. This is well defined by Proposition \ref{prop:psi_hat_tilde}. Let $\gamma\colon T_1\to T_2$ be a morphism, from $\alpha\colon T_1\to S$ to $\beta\colon T_2\to S$, in ${\mathcal{P}}(S)$. Put ${\widehat{G}}(\gamma) = \gamma|_{P(T_1)}\colon P(T_1)\to P(T_2)$. It is routine to verify that ${\widehat{G}}(\gamma)$ is well defined. That the assignment $\gamma\mapsto {\widehat{G}}(\gamma)$ is functorial is immediate. Thus ${\widehat{G}}$ is a functor.

		\begin{lemma}\label{lem:n2} If $\gamma\colon T_1 \to  T_2$ is surjective then ${\widehat{G}}(\gamma)$ satisfies (M3). 
		\end{lemma}
		\begin{proof} Suppose that $\gamma$ is a surjective morphism from $\alpha\colon T_1\to S$ to $\beta\colon T_2\to S$, and let $s\in S$. The proof amounts to verification of the inclusion
			${\operatorname{\mathrm{dom}}{\widehat\beta}_s} \, \cap \, \beta^{-1}(s^*) \subseteq \gamma({\operatorname{\mathrm{dom}}{\widehat\alpha}_s} \, \cap \, \alpha^{-1}(s^*))$. Let $y\in {\operatorname{\mathrm{dom}}{\widehat\beta}_s} \, \cap \, \beta^{-1}(s^*)$. Then $y\leq t^*$ for some $t\in T_2$ such that $\beta(t)\leq s$ and $\beta(y)=s^*$. Since $s^* = \beta(y)\leq \beta(t^*)\leq s^*$, we have $\beta(t^*)=s^*$. Thus $\beta(t)^*=s^*$. Since also $\beta(t)\leq s$, we have $\beta(t) = s$, by parts \eqref{i:b61} and \eqref{i:b8} of Lemma \ref{lem:lem1}.  Let $t_1=ty$. Then $\beta(t_1) = \beta(ty) = \beta(t)\beta(y) = ss^* =s$ and $y=t^*y=(ty)^*=t_1^*$. Because $\gamma$ is surjective, there is  $u \in T_1$ such that $\gamma(u)=t_1$. Observe that $\alpha(u) = \beta\gamma(u) = s$, thus $u^*\in {\operatorname{\mathrm{dom}}{\widehat\alpha}_s} \, \cap \, \alpha^{-1}(s^*)$, so that $y\in \gamma({\operatorname{\mathrm{dom}}{\widehat\alpha}_s} \, \cap \, \alpha^{-1}(s^*))$.
		\end{proof}
		
		\begin{theorem} \label{th:equivalence_of_cats}  Let $S$ be a restriction semigroup.  The functors ${\widehat{U}}\colon  {\widehat{\mathcal{A}}}(S)\to {\mathcal{P}}(S)$ and ${\widehat{G}}\colon {\mathcal{P}}(S) \to {\widehat{\mathcal{A}}}(S)$ establish an equivalence between the categories ${\widehat{\mathcal{A}}}(S)$ and ${\mathcal{P}}(S)$.
		\end{theorem}
		
		\begin{proof} 
			Let $(\varphi, p, X)$ be an object of ${\widehat{\mathcal{A}}}(S)$. We show that the map $$f_{(\varphi,p,X)}\colon X\to P(X\rtimes^p_{\varphi} S), \, x\mapsto (x,p(x)),$$ is an isomorphism, from $(\varphi, p, X)$ to ${\widehat{G}}{\widehat{U}}(\varphi, p, X)$, in ${\widehat{\mathcal{A}}}(S)$. We put ${\widehat{U}}(\varphi, p, X)=\Psi$ where $\Psi\colon X\rtimes^p_{\varphi} S \to S$, $(x,s)\mapsto s$. Then $${\widehat{G}}{\widehat{U}}(\varphi, p, X) = (\widehat{\Psi}, \widehat{\Psi}|_{P(X\rtimes^p_{\varphi} S)}, P(X\rtimes^p_{\varphi} S)).$$ 
			Clearly, $f_{(\varphi,p,X)}$ is an isomorphism of semilattices. We are left to show that 
			$$\operatorname{\mathrm{dom}}\widehat{\Psi}_s = f_{(\varphi,p,X)}(\operatorname{\mathrm{dom}}\varphi_s) \text{ and } \widehat{\Psi}_s(f_{(\varphi,p,X)}(x)) =f_{(\varphi,p,X)}(\varphi_s(x))$$ 
			for all $s\in S$ and $x\in \operatorname{\mathrm{dom}}\varphi_s$.
			Let $f_{(\varphi,p,X)}(x)=(x, p(x)) \in \operatorname{\mathrm{dom}}\widehat{\Psi}_s$. By  the definition of $ \widehat{\Psi}_s$ this means that $(x, p(x))\leq (y,r)^*$ where $\Psi(y,r)\leq s$. Since $(y,r)\in X\rtimes^p_{\varphi} S$, we also have $y\in \operatorname{\mathrm{ran}}\varphi_r$ and $p(y)=r^+$. Applying~\eqref{eq:def_Psi}, \eqref{eq:def_star_plus} and Lemma \ref{lem:properties_extension}, we obtain: $x\leq \varphi_r^{-1}(y)$, $p(x)\leq r^*$ and $r\leq s$. It follows that $x\in \operatorname{\mathrm{dom}}\varphi_r$. Since, by Lemma \ref{lem:prop_hat_tilde1}, $\varphi$ is order-preserving, it follows that $x\in \operatorname{\mathrm{dom}}\varphi_s$.
			Conversely, suppose that $x\in \operatorname{\mathrm{dom}}\varphi_s$ and  let us show that $f_{(\varphi,p,X)}(x)=(x, p(x)) \in \operatorname{\mathrm{dom}}\widehat{\Psi}_s$. By (A3a) we have that $x\in \operatorname{\mathrm{dom}}\varphi_t\cap p^{-1}(t^*)$ for some $t\leq s$. We have $\varphi_t(x)\in \operatorname{\mathrm{ran}}\varphi_t$ and, by Lemma \ref{lem:q}\eqref{i:lem:q1}, $\varphi_t(x) \in p^{-1}(t^+)$ so that
			$(\varphi_t(x), t)\in X\rtimes^p_{\varphi} S$.  Since $(x, t^*) = (\varphi_t(x), t)^*$ and $\Psi(\varphi_t(x), t) = t\leq s$, it follows that $f_{(\varphi,p,X)}(x)=(x, t^*)\in \operatorname{\mathrm{dom}}\widehat{\Psi}_s$, as needed. Finally, let $(x, p(x)) \in \operatorname{\mathrm{dom}}\widehat{\Psi}_s = f_{(\varphi,p,X)}(\operatorname{\mathrm{dom}}\varphi_s)$. Take $(y,r)\in X\rtimes^p_{\varphi} S$ such that $(x, p(x))\leq (y,r)^*$ where $r\leq s$.
			Then
			\begin{align*}\widehat{\Psi}_s (f_{(\varphi,p,X)}(x)) & = \widehat{\Psi}_s (x, p(x)) = ((y,r)(x,p(x)))^+ & \\
			& = (\varphi_r(\varphi_r^{-1}(y)\wedge x), rp(x))^+  &  \\
			& = (\varphi_r(x), (rp(x))^+) & (\text{since } x\leq \varphi_r^{-1}(y))\\
			& = f_{(\varphi,p,X)}(\varphi_r(x)) &(\text{by Lemma \ref{lem:q}(1)}) \\
			& = f_{(\varphi,p,X)}(\varphi_s(x)). &(\text{by Lemma \ref{lem:prop_hat_tilde1}(1)})
			\end{align*}
			Moreover, if $(\varphi_i, p_i, X_i)$, $i=1,2$, are objects of ${\widehat{\mathcal{A}}}(S)$ and $\gamma\colon X_1\to X_2$ is a morphism from $(\varphi_1, p_1, X_1)$ to $(\varphi_2, p_2, X_2)$ then
			${\widehat{G}}{\widehat{U}}(\gamma)\colon P(X_1\rtimes^{p_1}_{\varphi_1} S) \to P(X_2\rtimes^{p_2}_{\varphi_2} S)$ is given by  $(x,p_1(x))\mapsto (\gamma(x), p_1(x)) = (\gamma(x), p_2(\gamma(x)))$ whence the isomorphism $f_{(\varphi, p, X)}$ is easily seen to be natural in $(\varphi, p, X)$.

			In the reverse direction, let $\psi\colon T\to S$ be an object of the category ${\mathcal{P}}(S)$. Let us abbreviate $\psi\colon T\to S$ by $(\psi, T)$. Set  $p=\psi|_{P(T)}$. Then ${\widehat{U}}{\widehat{G}}(\psi,T)$ is the proper morphism $\Psi_{(\psi, T)}\colon P(T)\rtimes^p_{\widehat{\psi}} S \to S$, $(e,s)\mapsto s$. The proof of Theorem~\ref{th:isom_epsilon}  implies that the map ${\widehat{\eta}}_{(\psi, T)} \colon T \to P(T)\rtimes^p_{\widehat{\psi}} S$ thereof is an isomorphism from $(\psi, T)$ to $(\Psi_{(\psi, T)}, P(T)\rtimes^p_{\widehat{\psi}} S)$ in the category ${\mathcal{P}}(S)$. Moreover, if $\gamma\colon T_1\to T_2$ is a morphism from $(\psi_1, T_1)$ to $(\psi_2, T_2)$ in the category ${\mathcal{P}}(S)$ then ${\widehat{U}}{\widehat{G}}(\gamma)\colon (\Psi_{(\psi_1, T_1)}, P(T_1)\rtimes^{p_1}_{\widehat{\psi_1}} S) \to (\Psi_{(\psi_2, T_2)}, P(T_2)\rtimes^{p_2}_{\widehat{\psi_2}} S)$ is given by $(e,s)\mapsto (\gamma(e),s)$.  Hence ${\widehat{\eta}}_{(\psi, T)}$ is natural in $(\psi, T)$.
			This completes the proof. 
		\end{proof}
		
		We put ${\widetilde{G}} = {\widetilde{F}}{\widehat{I}}{\widehat{G}}\colon {\mathcal{P}}(S) \to {\widetilde{\mathcal{A}}}(S)$.   Theorem \ref{th:adjunctions} implies the following.
		
		\begin{corollary}\label{cor:cor12} Let $S$ be a restriction semigroup. 
			\begin{enumerate}[(1)]
				\item The functors ${\widetilde{U}}\colon  {\widetilde{\mathcal{A}}}(S)\to {\mathcal{P}}(S)$ and ${\widetilde{G}}\colon {\mathcal{P}}(S) \to {\widetilde{\mathcal{A}}}(S)$ establish an equivalence between the categories ${\widetilde{\mathcal{A}}}(S)$ and ${\mathcal{P}}(S)$.
				\item The functor $U\colon {\mathcal{A}}(S)\to {\mathcal{P}}(S)$ is a left adjoint to the functor ${\widehat{I}}{\widehat{G}}\colon {\mathcal{P}}(S) \to  {\mathcal{A}}(S)$.
				\item The functor $U\colon {\mathcal{A}}(S)\to {\mathcal{P}}(S)$ is a right adjoint to the functor ${\widetilde{I}}{\widetilde{G}}\colon {\mathcal{P}}(S) \to  {\mathcal{A}}(S)$.
			\end{enumerate}
		\end{corollary}
		
		${\mathcal{P}}_s(S)$ be the subcategory of the category ${\mathcal{P}}(S)$ whose morphisms are surjective. By Lemmas  \ref{lem:n1} and \ref{lem:n2}, restricting the functors of Corollary \ref{cor:cor12} we obtain that each of the categories  ${\widetilde{\mathcal{A}}}_s(S)$ and  ${\widehat{\mathcal{A}}}_s(S)$ is equivalent to the category ${\mathcal{P}}_s(S)$, and also that there are adjunctions between the categories ${\mathcal{A}}_s(S)$ and ${\mathcal{P}}_s(S)$.
		
		For our final result in this section, we assume that $S$ is a monoid.  In this case the  objects of the categories ${\mathcal{P}}(S)$ and ${\mathcal{P}}_s(S)$ are  proper restriction semigroups $T$ satisfying $S\simeq T/\sigma$. Taking into account Remark \ref{rem:trivial}, we obtain the following statement.
		
		\begin{corollary} Let $S$ be a monoid. 
			\begin{enumerate}
				\item The categories ${\mathcal{P}}(S)$ and ${\mathcal{A}}(S)$ are equivalent.
				\item The categories ${\mathcal{P}}_s(S)$ and ${\mathcal{A}}_s(S)$ are equivalent.
			\end{enumerate}
		\end{corollary}
		
		We remark that the category ${\mathcal{P}}(S)$ with $S$ being an inverse semigroup was considered in \cite{Lawson96} (denoted by ${\mathbf{IP}}_S$ in \cite{Lawson96}) where an equivalence of the category ${\mathcal{P}}(S)$ and a certain category of effective ordered coverings of $S$ was established (\cite[Theorem 4.4]{Lawson96}). This leads to a connection between the categories ${\widetilde{\mathcal{A}}}(S)$,  ${\widehat{\mathcal{A}}}(S)$, ${\mathcal{A}}(S)$ and the  categories studied in~\cite{Lawson96}.
		
		\section{$F$-morphisms}\label{s:f}
		
		Let $\psi\colon T\to S$ be a proper morphism between restriction semigroups. In this section we consider the situation where $\psi^{-1}(s)$ has a maximum element for each $s\in S$.  Then a map $\tau\colon S\to T$ can be defined by 
		\begin{equation}\label{eq:tau}
		\tau(s) = {\mathrm{max}}\{t\in T\colon \psi(t)=s\}.
		\end{equation}
		It is easy to see that $\tau$ is a premorphism. 
		
		\begin{proposition}\label{prop:7.1} Let $\psi\colon T\to S$ be a surjective morphism between restriction semigroups such that $\psi^{-1}(s)$ has a maximum element for each $s\in S$. Then $\psi$ is proper.
		\end{proposition}
		
		\begin{proof}  Let $s,t\in  T$ and assume that $\psi(s)=\psi(t)$. Then $s\leq \tau\psi(s) = \tau\psi(t)$ and $t\leq \tau\psi(t)$. It follows that $s\sim t$, so that $\psi$ is proper.
		\end{proof}
		
		Note that, under the conditions of Proposition \ref{prop:7.1}, for each $s\in S$ we have
		$$
		\operatorname{\mathrm{dom}}\widetilde{\psi}_s = (\tau(s)^{*})^{\downarrow}  \text{ and } \widetilde{\psi}_s(e) = (\tau(s)e)^+, e\leq \tau(s)^{*}.
		$$
		It follows that $\widetilde{\psi} = \alpha\tau$ where $\alpha$ is the Munn representation of $T$.
		
		\begin{proposition}  Let $\psi\colon T\to S$ be a proper morphism and assume that each $\psi$-class has a maximum element. Let $\tau \colon S\to T$  be the premorphism defined in \eqref{eq:tau}. Then $\tau$ is order-preserving (resp. locally strong, strong, locally multiplicative, multiplicative) if and only if $\widetilde{\psi}$ is order-preserving (resp. locally strong, strong, locally multiplicative, multiplicative).
		\end{proposition}
		
		\begin{proof}
			Let $s,t\in S$ be such that $s\leq t$. Condition $\widetilde{\psi}_s \leq \widetilde{\psi}_t$ is equivalent to $\widetilde{\psi}_s = \widetilde{\psi}_t\widetilde{\psi}_s^{\,*}$. This can be rewritten
			as $\alpha\tau(s) = \alpha\tau(t)(\alpha\tau(s))^*$ or, since $\alpha$ is a morphism, as 
			\begin{equation}\label{eq:aux26a}
			\alpha(\tau(s)) = \alpha(\tau(t)\tau(s)^*).
			\end{equation}
			Because $s= \psi(\tau(s)) = \psi(\tau(t)\tau(s)^*)$, we have $\tau(s)\sim \tau(t)\tau(s)^*$, as $\psi$ is proper.  The definition of $\alpha$ yields $(\tau(s)^*)^{\downarrow}=\operatorname{\mathrm{dom}}\alpha(\tau(s)) =  \operatorname{\mathrm{dom}}\alpha(\tau(t)\tau(s)^*) = ((\tau(t)\tau(s)^*)^*)^{\downarrow}$.
			It follows that $\tau(s)=\tau(t)\tau(s)^*$.
			We have shown that $\tau$ is order-preserving if and only if so is 
			$\widetilde{\psi}$.
			
			Let $s,t\in S$. The equality $\widetilde{\psi}_s\widetilde{\psi}_{s^*t}=  \widetilde{\psi}_s^{\,+}\widetilde{\psi}_{st}$ can be rewritten in the form 
			\begin{equation}\label{eq:aux26}
			\alpha(\tau(s)\tau(s^*t))=  \alpha(\tau(s)^+\tau(st)).
			\end{equation}
			Observe that $\psi(\tau(s)\tau(s^*t)) = \psi(\tau(s)^+\tau(st)) = st$ so that $\tau(s)\tau(s^*t) \sim \tau(s)^+\tau(st)$. By the definition of $\alpha$ we have $((\tau(s)\tau(s^*t))^*)^{\downarrow}=\operatorname{\mathrm{dom}}\alpha(\tau(s)\tau(s^*t)) =  \operatorname{\mathrm{dom}}\alpha(\tau(s)^+\tau(st)) = ((\tau(s)^+\tau(st))^*)^{\downarrow}$, thus $\tau(s)\tau(s^*t)=\tau(s)^+\tau(st)$. Hence $\tau(s)\tau(s^*t) = \tau(s)^+\tau(st)$.  It follows that $\widetilde{\psi}$ satisfies (LSl) if and only if so does $\tau$. Likewise, $\widetilde{\psi}$ satisfies (LSr) if and only if so does $\tau$. 
			Thus $\widetilde{\psi}$ is locally strong if and only if so is $\tau$.
			Similarly, $\widetilde{\psi}$ is  locally multiplicative if and only if so is $\tau$.
			
			Applying Theorem \ref{th:strong_restr}  (resp. Proposition \ref{prop:local2}), it follows that $\widetilde{\psi}$ is strong (resp. multiplicative) if and only if so is $\tau$.
		\end{proof}
		
		Following \cite{LMS}, we call $\psi$ an $F$-{\em morphism} provided that $\tau$ is order-preserving. If, in addition, $\tau$ is locally strong (or, equivalently, strong, see Theorem \ref{th:strong_restr}) we call $\psi$ an $F\!A$-{\em morphism} (adopting the terminology from \cite{Gomes06, GH09} where a similar notion for left restriction semigroups has been considered). Observe that $F$-morphisms are more general then $F\!A$-morphisms: $F$-morphisms generalize $F$-restriction monoids from \cite{Jones16,Kud15} whereas $F\!A$-morphisms generalize extra proper $F$-restriction monoids of \cite{Kud15}. However, if $S$ and $T$ are inverse semigroups, the difference between  $F$-morphisms and $F\!A$-morphisms vanishes, as both of them specialize to $F$-morphisms between inverse semigroups from \cite{LMS}. Finally, we remark that perfect $F$-morphisms generalize perfect $F$-restriction monoids \cite{Jones16,Kud15} (termed ultra $F$-restriction monoids in \cite{Kud15}).
		
		\section*{Acknowledgements} The work on this paper began during the visits of the first author (in January 2017) and of the second author (in January 2017,  February 2018 and  February 2019) to the Institute of Mathematics, Physics and Mechanics (IMFM) of Ljubljana, some of which were partially supported by ARRS grant P1-0288. The first author was  partially supported by FAPESP of Brazil (Process 2015/09162-9) and by CNPq of Brazil (Process 307873/2017-0).} The second author was partially supported by CNPq of Brazil (Process 404649/2018-1) and by the Funda\c{c}\~ao para a Ci\^{e}ncia e a Tecnologia (Portuguese Foundation for Science and Technology) through the project PTDC/MAT-PUR/31174/2017. The third author was partially supported by ARRS grant P1-0288.
		
		We thank the referee for a very careful reading of the paper and a number of valuable comments and suggestions.

\end{document}